\documentclass[11pt,a4paper]{article}

\usepackage{amsmath}
\usepackage{mathtools}
\usepackage{amsthm}
\usepackage{amssymb}
\usepackage{cain_arxiv}

\usepackage{hyperref}

\newcommand*\fullref[3][\relax]{%
  \ifdefined\hyperref%
    {\hyperref[#3]{#2\penalty 200\ \ref*{#3}#1}}%
  \else%
    {#2\penalty 200\ \relax\ref{#3}#1}%
  \fi%
}

\theoremstyle{definition}
\newtheorem{definition}{Definition}[section]
\newtheorem{algorithm}[definition]{Algorithm}
\newtheorem{example}[definition]{Example}

\newtheorem{question}[definition]{Question}

\theoremstyle{plain}
\newtheorem{corollary}[definition]{Corollary}
\newtheorem{lemma}[definition]{Lemma}
\newtheorem{proposition}[definition]{Proposition}
\newtheorem{theorem}[definition]{Theorem}

\numberwithin{equation}{section}

\newcommand*{\defterm}[1]{\textit{#1}}

\DeclarePairedDelimiter{\set}{\{}{\}}
\DeclarePairedDelimiterX{\gset}[2]{\{}{\}}{\,#1:#2\,}
\newcommand\gsetsplit[3][]{\mathopen#1\{\,#2:#3\,\mathclose#1\}}

\DeclareMathOperator{\im}{im}

\newcommand*{\emptyword}{\varepsilon}
\newcommand*{\imreduces}{\rightarrow}

\newcommand*{\reduces}{\rightarrow^*}

\newcommand*{\thue}{\leftrightarrow^*}

\newcommand*{\nset}{\mathbb{N}}

\DeclarePairedDelimiterX{\pres}[2]{\langle}{\rangle}{#1\,\delimsize\vert\,\mathopen{}#2}

\newcommand*{\drel}[1]{\mathcal{#1}}
\newcommand*{\cgen}[1]{#1^{\#}}

\newcommand*{\gH}{\mathrel{\mathcal{H}}}
\newcommand*{\gL}{\mathrel{\mathcal{L}}}
\newcommand*{\gR}{\mathrel{\mathcal{R}}}
\newcommand*{\gD}{\mathrel{\mathcal{D}}}
\newcommand*{\gJ}{\mathrel{\mathcal{J}}}

\newcommand*{\lpad}{\delta_{\mathrm{L}}}
\newcommand*{\rpad}{\delta_{\mathrm{R}}}

\newcommand{\fgrp}[1]{\mathrm{FG}\!\left( #1 \right)}

\newcommand{\elt}[1]{\overline{#1}}

\newcommand{\bigO}{\mathcal{O}}

\newcommand{\derives}{\Rightarrow^*}

\newcommand{\rev}{\mathrm{rev}}

\newcommand{\cliff}{\mathcal{S}}

\newcommand{\Yes}{\emph{Yes}}
\newcommand{\No}{\emph{No}}

\begin{document}

\title{Decision problems for word-hy\-per\-bol\-ic semigroups}
\author{Alan J. Cain \& Markus Pfeiffer}
\date{}

\thanks{During the research that led to the this paper, the first author was initially supported by the European
  Regional Development Fund through the programme {\sc COMPETE} and by the Portuguese Government through the {\sc FCT}
  (Funda\c{c}\~{a}o para a Ci\^{e}ncia e a Tecnologia) under the project {\sc PEst-C}/{\sc MAT}/{\sc UI}0144/2011 and
  through an {\sc FCT} Ci\^{e}ncia 2008 fellowship, and later supported by an {\sc FCT} Investigador advanced fellowship
  ({\sc IF}/01622/2013/{\sc CP}1161/{\sc CT}0001). This work was partially supported by {\sc FCT} through the project
  {\sc UID}/{\sc MAT}/00297/2013 (Centro de Matem\'{a}tica e Aplica\c{c}\~{o}es). Part of the work described here was
  carried out during a visit by the second author to the Universidade Nova de Lisboa, which was funded by an {\sc
    EPSRC} Doctoral Prize 2012, and during a visit by the first author to the University of St Andrews, which was funded
  by a London Mathematical Society Research in Pairs Grant (ref.~41410).}

\maketitle

\address[AJC]{%
Centro de Matem\'{a}tica e Aplica\c{c}\~{o}es, Faculdade de Ci\^{e}ncias e Tecnologia \\
Universidade Nova de Lisboa, 2829--516 Caparica, Portugal
}
\email{%
a.cain@fct.unl.pt
}

\address[MP]{%
School of Mathematics \& Statistics, University of St Andrews \\
North Haugh, St Andrews, Fife, KY16 9SX, United Kingdom
}
\email{%
markus.pfeiffer@st-andrews.ac.uk
}

\begin{abstract}
  This paper studies decision problems for semigroups that are word-hy\-per\-bol\-ic in the sense of Duncan \& Gilman. A
  fundamental investigation reveals that the natural definition of a `word-hy\-per\-bol\-ic structure' has to be
  strengthened slightly in order to define a unique semigroup up to isomorphism. The isomorphism problem is proven to be
  undecidable for word-hyperbolic semigroups (in contrast to the situation for word-hyperbolic groups). It is proved
  that it is undecidable whether a word-hyperbolic semigroup is automatic, asynchronously automatic, biautomatic, or
  asynchronously biautomatic. (These properties do not hold in general for word-hyperbolic semigroups.) It is proved
  that the uniform word problem for word-hy\-per\-bol\-ic semigroup is solvable in polynomial time (improving on the
  previous exponential-time algorithm). Algorithms are presented for deciding whether a word-hy\-per\-bol\-ic semigroup
  is a monoid, a group, a completely simple semigroup, a Clifford semigroup, or a free semigroup.
\end{abstract}


\section{Introduction}

The concept of word-hy\-per\-bol\-ic\-ity in groups, which has grown into one
of the most fruitful areas of group theory since the publication of
Gromov's seminal paper \cite{gromov_hyperbolic}, admits a natural
extension to monoids via using Gilman's characterization of
word-hy\-per\-bol\-ic groups using context-free languages
\cite{gilman_hyperbolic}, which generalizes directly to semigroups and
monoids \cite{duncan_hyperbolic}. Informally, a word-hy\-per\-bol\-ic
structure for a semigroup consists of a regular language of
representatives (not necessarily unique) for the elements of the
semigroup, and a context-free language describing the multiplication
table of the semigroup in terms of those representatives.

This generalization has led to a substantial amount of research on
word-hy\-per\-bol\-ic semigroups; see, for example,
\cite{cm_wordhypunique,fountain_hyperbolic,hoffmann_relatives,hoffmann_notions}. Some
of this work has shown that word-hy\-per\-bol\-ic semigroups do not possess
such pleasant properties as word-hy\-per\-bol\-ic groups: they may not be
finitely presented, and they are not in general automatic or even
asynchronously automatic \cite[Example~7.7
  et~seq.]{hoffmann_relatives}.

The computational aspect of word-hy\-per\-bol\-ic semigroups has so far
received limited attention. The only established result seems to be
the solvablity of the word problem
\cite[Theorem~3.8]{hoffmann_relatives}. In contrast, automatic
semigroups, which generalize automatic groups \cite{epstein_wordproc}
and whose study was inaugurated by Campbell et
al.~\cite{campbell_autsg}, have been studied from a computational
perspective, with both decidability and undecidability results
emerging \cite{c_cancundec,kambites_decision,otto_green}.

This paper is devoted to some important decision problems for
word-hy\-per\-bol\-ic semigroups. Word-hy\-per\-bol\-ic structures are not
necessarily `stronger' or `weaker' computationally than automatic
structures. As noted above, word-hy\-per\-bol\-ic\-ity does not imply
automaticity for semigroups, so one cannot appeal to known results for
automatic semigroups. A word-hy\-per\-bol\-ic structure encodes the whole
multiplication table for the semigroup, not just right-multiplication
by generators (as is the case for automatic structures). On the other
hand, context-free languages are computationally less pleasant than
regular languages. For instance, an intersection of two context-free
languages is not in general context-free, and indeed the emptiness of
such an intersection cannot be decided algorithmically. Thus, in
constructing algorithms for word-hy\-per\-bol\-ic semigroups, it is often
necessary to proceed via an indirect route, or use some unusual
`trick'.

Two of the most important results in this paper are the undecidability results in
\fullref{Section}{sec:isoundec}. First, the isomorphism problem for word-hyperbolic semigroups is undecidable, which
contrasts the decidability of the isomorphism problem for hyperbolic groups
\cite[Theorem~1]{dahmani_isomorphism}. Second, it is undecidable whether a word-hyperbolic semigroup is automatic. (As
noted above, for semigroups, word-hyperbolicity does not in general imply automaticity.)

Among the positive decidability results, the most important is that the uniform word problem for word-hyperbolic
semigroups is soluble in polynomial time (\fullref{Section}{sec:wordproblem}). As remarked above, the word problem was
already known to be solvable, but the previously-known algorithm required time exponential in the lengths of the input
words \cite[Theorem~3.8]{hoffmann_relatives}.

Some basic properties are then shown to be decidable
(\fullref{Section}{sec:basic}): being a monoid, Green's relations
$\gL$, $\gR$, and $\gH$, being a group, and commutativity. These
results are not particularly difficult, but are worth noting.

The main body of the paper shows the decidability of more
complicated algebraic properties: being completely simple
(\fullref{Section}{sec:compsimp}), being a Clifford semigroup
(\fullref{Section}{sec:clifford}), and being a free semigroup
(\fullref{Section}{sec:free}).

Before embarking on the discussion of decision problems, it is
necessary to make a fundamental study of the notion of
word-hy\-per\-bol\-ic\-ity, because the natural notion of a word-hy\-per\-bol\-ic
structure, or more precisely an `interpretation' of a word-hy\-per\-bol\-ic
structure, does not determine a unique semigroup up to isomorphism. A
slightly strengthened definition is needed, and this is the purpose of
the preliminary \fullref{Section}{sec:interpretations}.

The paper ends with a list of some open problems
(\fullref{Section}{sec:questions}).

\section{Preliminaries}

Throughout the paper, we assume basic knowledge of regular language and finite state automata, of context-free languages
and pushdown automata, and of rational relations and transducers; see \cite{hopcroft_automata} and
\cite{berstel_transductions} for background reading.

We denote the empty word (over any alphabet) by $\emptyword$. For an alphabet $A$, we denote by $A^*$ the language of
all words over $A$, and by $A^+$ the language of all non-empty words over $A$. The length of $u \in A^*$ is denoted
$|u|$, and, for any $a \in A$. We denote by $u^\rev$ the reversal of a word $u$; that is, if $u=a_1\cdots a_{n-1}a_n$
then $u^\rev = a_na_{n-1}\cdots a_1$, with $a_i\in A$. We extend this notation to languages: for any language $L
\subseteq A^*$, let $L^\rev = \gset{w^\rev}{w \in L}$.

If $\drel{R}$ is a relation on $A^*$, then $\cgen{\drel{R}}$ denotes the congruence generated by $\drel{R}$. A presentation
is a pair $\pres{A}{\drel{R}}$ that defines [any semigroup isomorphic to] $A^+/\cgen{\drel{R}}$.

\section{The limits of interpretation}
\label{sec:interpretations}

Before developing any algorithms for word-hy\-per\-bol\-ic semigroups, we must clarify the relationship between a
word-hy\-per\-bol\-ic structure (that is, an abstract collection of certain languages) and a semigroup it describes. A
similar study grounds the study of decision problems for automatic semigroups by Kambites \& Otto
\cite{kambites_decision}, and our strategy and choice of terminology closely follows theirs.

\begin{definition}
A \defterm{pre-word-hy\-per\-bol\-ic structure} $\Sigma$ consists of:
\begin{itemize}

\item a finite alphabet $A(\Sigma)$;

\item a regular language $L(\Sigma)$ over $A(\Sigma)$, not including the empty word;

\item a context-free language $M(\Sigma)$ over $A(\Sigma) \cup
  \set{\#_1,\#_2}$, where $\#_1$ and $\#_2$ are new symbols not in
  $A(\Sigma)$, such that $M(\Sigma) \subseteq L(\Sigma)\#_1L(\Sigma)\#_2L(\Sigma)^\rev$.
\end{itemize}
When $\Sigma$ is clear from the context, we may write $A$, $L$, and
  $M$ instead of $A(\Sigma)$, $L(\Sigma)$, and $M(\Sigma)$,
  respectively.
\end{definition}

The idea is that $A(\Sigma)$ will represent a set of generators for a semigroup, $L(\Sigma)$ will be a language of
representatives for the elements of that semigroup, and $M(\Sigma)$ will describe the multiplication table for that
semigroup in terms of the representatives in $L(\Sigma)$. However, a `pre-word-hy\-per\-bol\-ic structure' consists only
of languages fulfilling certain basic properties: there is no mention of being a structure `for a semigroup' in the
definition. In particular, at this point there is nothing that guarantees $L(\Sigma)$ or $M(\Sigma)$ are non-empty. Or,
$L(\Sigma)$ could be $A(\Sigma)^+$ and $M(\Sigma)$ could be the language $u\#_1v\#_2L(\Sigma)$ for some fixed $u,v \in
L(\Sigma)$; clearly, $M(\Sigma)$ is very far from describing a multiplication table.

Now, following Kambites \& Otto for automatic semigroups \cite[\S~2.2]{kambites_decision}, let us make a first attempt
to turn the abstract pre-word-hy\-per\-bol\-ic structure into something that describes a semigroup. As we shall see,
this definition is flawed, but it is instructive to see its consequences, since these illustrate why the improved
\fullref{Definition}{def:improved} is actually the correct one.

\begin{definition}[(First attempt)]
\label{def:first}
An \defterm{interpretation} of a pre-word-hy\-per\-bol\-ic structure
$\Sigma$ with respect to a semigroup $S$ is a homomorphism $\phi : A^+
\to S$ such that $L\phi = S$ and
\[
M(\Sigma) = \gset[\big]{u\#_1v\#_2w^\rev}{u,v,w \in L, (u\phi)(v\phi) = w\phi}.
\]
When there is no risk of confusion, denote $u\phi$ by $\elt{u}$ for
  any $u \in A^+$, and $X\phi$ by $\elt{X}$ for any $X \subseteq A^+$.

If a pre-word-hy\-per\-bol\-ic structure $\Sigma$ admits an interpretation
with respect to a semigroup $S$, then $\Sigma$ is a
\defterm{word-hy\-per\-bol\-ic structure} for $S$.

A semigroup is \defterm{word-hyperbolic} if it admits a word-hyperbolic structure.
\end{definition}

Suppose $\Sigma$ is a word-hy\-per\-bol\-ic structure for $S$, as per
\fullref{Definition}{def:first}. Then words in $A^+$ represent
elements of $S$, the regular language $L$ contains at least one
representative for every element of $S$, and the context-free
language $M$ encodes the multiplication table for $S$ in terms of
representatives in $L$. However, there is a problem. In contrast to
the situation for automatic semigroups
\cite[Proposition~2.3]{kambites_decision}, a word-hy\-per\-bol\-ic structure
(using \fullref{Definition}{def:first}) does not uniquely determine a
semigroup: the same pre-word-hy\-per\-bol\-ic structure can admit
interpretations with respect to non-isomorphic semigroups, as the
following example shows:

\begin{example}
\label{ex:wordhypdoesnotdetermine}
Let $\Sigma$ be a pre-word-hy\-per\-bol\-ic structure with $A = \set{a,b,c}$,
$L = A$, and $M = \gset{u\#_1v\#_2a^\rev}{u,v \in L}$. (Of course,
$a^\rev = a$ since $a$ is a single letter.)

Let $S$ be the two-element null semigroup $\set{0,x}$, where all
products are equal to $0$. Let $T$ be the three-element null semigroup
$\set{0,x,y}$, again with all products equal to $0$.

Define mappings $\phi : A \to S$ and $\psi : A \to T$ by
\begin{align*}
a\phi &= 0, & b\phi &= x, & c\phi &= x,\\
a\psi &= 0, & b\psi &= x, & c\psi &= y.
\end{align*}
Then $L\phi = S$ and $L\psi = T$. Furthermore,
\begin{align*}
M &= \gset{u\#_1v\#_2a^\rev}{u,v \in L} \\
&= \gset{u\#_1v\#_2a^\rev}{u,v \in L, (u\phi)(v\phi) = a\phi} \\
&= \gset{u\#_1v\#_2w^\rev}{u,v,w \in L, (u\phi)(v\phi) = w\phi},
\end{align*}
since all products in $S$ are equal to $0$ and $a$ is the unique word
in $L$ mapped to $0$ by $\phi$. Similarly,
\[
M = \gset[\big]{u\#_1v\#_2w^\rev}{u,v,w \in L, (u\psi)(v\psi) = w\psi}
\]
since all products in $T$ are equal to $0$ and $a$ is the unique word
in $L$ in mapped to $0$ by $\psi$.

Thus $\phi$ and $\psi$ are interpretations of $\Sigma$ with respect
to the non-isomorphic semigroups $S$ and $T$ respectively.
\end{example}

Hence, it seems not to make sense to consider decision problems for general word-hy\-per\-bol\-ic semigroups, at least
with the current definitions. It would be illogical to ask for an algorithm that takes as input a word-hy\-per\-bol\-ic
structure and determined some property of `the' semigroup is describes, since there is no such unique semigroup. The
fundamental problem \fullref{Example}{ex:wordhypdoesnotdetermine} elucidates is that the word-hy\-per\-bol\-ic structure
$\Sigma$ does not necessarily determine whether the two symbols $b$ and $c$ represent the same element or different
elements. However, this problem only arises for two symbols in $A$ (that is, for two words over $A$ of length $1$). This
is because, in a sense we shall formalize shortly, a word-hyperbolic structure $\Sigma$ does determine if two words in
$L(\Sigma)$ represent the same \emph{decomposable} element of a semigroup: Suppose that $\phi : A(\Sigma)^+ \to S$ is an
interpretation of $\Sigma$, and let $w,x \in L(\Sigma)$ be such that $w\phi$ is decomposable as $(u\phi)(v\phi)$. Then
\[
w\phi = x\phi \iff (\exists u,v \in L(\Sigma))\bigl(u\#_1v\#_2w^\rev \in M(\Sigma) \land u\#_1v\#_2x^\rev \in M(\Sigma)\bigr).
\]
That is, one can tell whether $w$ and $x$ represent the same element of the semigroup by using the information in
$\Sigma)$ to factor $w\phi$ as $u\phi$ and $v\phi$ and check if $w\phi = (u\phi)(v\phi) = x\phi$. In
\fullref{Lemma}{lem:equality} below, we shall show consider a relation $E(\Sigma)$ that encapsulates the information
about representatives for the same element that can be extracted from $\Sigma$ by this `decompostion and multiplication'
trick.

Note that because all elements of a monoid are decomposable, the relation $E(\Sigma)$ relates \emph{all} pairs of words
that represent the same element. However, for general semigroups, one needs to go further.

Since words of length at least $2$ in $L(\Sigma)$ must represent decomposable elements, the only ambiguity is in whether
letters in $A(\Sigma)$ represent equal elements of $S$. Thus, as we shall see, if we insist that an interpretation
should consist of a map that sends distinct letters to distinct elements of the semigroup (that is, the homomorphism is
an injection when restricted to $A(\Sigma)$), then a word-hy\-per\-bol\-ic structure does describe a unique semigroup up
to isomorphism. In order to prove this result (\fullref{Proposition}{prop:isomorphic}), we need the following lemma:

\begin{lemma}
\label{lem:equality}
Let $\Sigma$ be a word-hy\-per\-bol\-ic structure. Then there is a relation
$E(\Sigma) \subseteq L(\Sigma) \times L(\Sigma)$, dependent only on
$\Sigma$, such that the following are equivalent for any words $w,x \in L(\Sigma)$:
\begin{enumerate}
\item $(w,x) \in E(\Sigma)$;
\item $w\phi = x\phi$ for some interpretation $\phi : A(\Sigma)^+ \to S$ with
  $\phi|_A$ being an injection;
\item $w\phi = x\phi$ for any interpretation $\phi : A(\Sigma)^+ \to S$ with
  $\phi|_A$ being an injection.
\end{enumerate}
\end{lemma}

\begin{proof}
Define
\begin{align*}
E' &= \gsetsplit[\big]{(w,x)}{w \in L(\Sigma), x \in L(\Sigma), |w| \geq |x|, |w| \geq 2,\\
&\qquad\qquad (\exists u,v \in L(\Sigma))(u\#_1v\#_2w^\rev \in M(\Sigma) \land u\#_1v\#_2x^\rev \in M(\Sigma)} \\
\end{align*}
and let
\begin{equation}
\label{eq:esigma}
E(\Sigma) = \gset[\big]{(a,a)}{a \in A(\Sigma) \cap L(\Sigma)} \cup E' \cup (E')^{-1}.
\end{equation}
The aim is now to show that $E(\Sigma)$ has the required
properties. Let $w, x \in L(\Sigma)$.

First suppose that (1) holds; that is, that $(w,x) \in E(\Sigma)$. Let
$\phi$ be any interpretation of $\Sigma$. Either $w,x \in A(\Sigma)
\cap L(\Sigma)$, in which case $w = x$ by \eqref{eq:esigma} and so
$w\phi = x\phi$, or $(w,x) \in E' \cup (E')^{-1}$. Assume $(w,x) \in
E'$; the other case is symmetrical. Then there exist $u,v \in L(\Sigma)$ such
that $u\#_1v\#_2w^\rev \in M(\Sigma)$ and $u\#_1v\#_2x^\rev \in
M(\Sigma)$. Hence $w\phi = (u\phi)(v\phi) = x\phi$ since $\phi$ is an
interpretation of $\Sigma$. Hence~(1) implies~(3).

It is clear that (3) implies (2). Now suppose that (2) holds; that is,
that $w\phi = x\phi$ for some interpretation $\phi : A(\Sigma)^+ \to
S$ with $\phi|_A$ being an injection. If $|w| = |x| = 1$, then $w,x
\in A$ and so $w = x$ since $\phi|_A$ is injective, and so $(w,x) \in
E(\Sigma)$. Now suppose that at least one of $|w|$ and $|x|$ is
greater than $1$. Assume $|w| \geq |x|$; the other case is
similar. Since $w$ has at least two letters, the element $w\phi$ is
decomposable in $S$. So there are words $u,v \in L$ with
$(u\phi)(v\phi) = w\phi$. Since $w\phi = x\phi$, it also follows that
$(u\phi)(v\phi) = x\phi$. Thus the words $u\#_1v\#_2w^\rev$ and
$u\#_1v\#_1x^\rev$ both lie in $M$ since $\phi$ is an interpretation
of $\Sigma$. Hence $(w,x) \in E' \subseteq E(\Sigma)$. Hence~(2)
implies~(1).
\end{proof}

\begin{proposition}
\label{prop:isomorphic}
Let $\Sigma$ be a word-hy\-per\-bol\-ic structure admitting interpretations
$\phi : A^+ \to S$ and $\psi : A^+ \to T$, with both $\phi|_A$ and
$\psi|_A$ being injections. Then there is an isomorphism $\tau$ from
$S$ to $T$ such that $\phi|_L\tau = \psi|_L$.
\end{proposition}

\begin{proof}
Define maps $\tau : S \to T$ and $\tau' : T \to S$ as follows. For any
$s \in S$ let $s\tau$ be $w\psi$, where $w \in L$ is some word with
$w\phi = s$, and for any $t \in T$, let $t\tau'$ be $w'\phi$, where
$w' \in L$ is some word with $w'\psi = t$. (The words $w$ and $w'$ are
guaranteed to exist since $\phi$ and $\psi$ are surjections.) The maps
$\tau$ and $\tau'$ are well-defined as a consequence of
\fullref{Lemma}{lem:equality}.

To show that $\tau$ is a homomorphism, proceed as follows. Let $r,s \in
S$ and choose $u,v,w \in L$ with $u\phi = r$, $v\phi = s$, and $w\phi
= rs$. Then $r\tau = u\psi$, $s\tau = v\psi$, and $(rs)\tau = w\psi$, by
the definition of $\tau$. Now, $u\#_1v\#_2w^\rev \in M$ (since $\phi$
is an interpretation of $\Sigma$) and so $(u\psi)(v\psi) = w\psi$
(since $\psi$ is an interpretation of $\Sigma$). Thus
\[
(r\tau)(s\tau) = (u\psi)(v\psi) = (w\psi) = (rs)\tau
\]
and so $\tau$ is a homomorphism.

Symmetric reasoning shows that $\tau' : T \to S$ is a
homomorphism. The maps $\tau$ and $\tau'$ are mutually inverse, since
if $w \in L$ is such that $w\phi = s$ and $w\psi = t$, then $s\tau =
t$ and $\tau' = s$. Thus $\tau : S \to T$ is an isomorphism. By the
definition of $\tau$ using elements of $L$, it follows that
$\phi|_L\tau = \psi|_L$.
\end{proof}

The extra condition used in \fullref{Proposition}{prop:isomorphic},
where the interpretation restricted to the alphabet $A$ is an
injection, does not restrict the class of word-hy\-per\-bol\-ic semigroups:

\begin{proposition}
\label{prop:wordhypstrinj}
Let $\Sigma$ be a word-hy\-per\-bol\-ic structure and $\phi : A(\Sigma)^+
\to S$ an interpretation for $\Sigma$ with respect to a semigroup
$S$. Then there is a word-hy\-per\-bol\-ic structure $\Pi$, effectively
computable from $\Sigma$ and $\phi|_{A(\Sigma)}$, with $A(\Pi)
\subseteq A(\Sigma)$, admitting an interpretation $\psi : A(\Pi)^+ \to
S$ with $\psi|_{A(\Pi)}$ being an injection.
\end{proposition}

\begin{proof}
Initially, let $\Pi = \Sigma$. We will modify $\Pi$ until it has the
desired property.

Suppose $\phi|_{A(\Pi)}$ is not injective. Pick $a,b \in A(\Pi)$ with
$a\phi = b\phi$. Replace every instance of $b$ by $a$ in words in
$L(\Pi)$. (This corresponds to replacing $b$ by $a$ whenever it
appears as a label on an edge in a finite automaton recognizing
$L(\Pi)$.) Replace every instance of $b$ by $a$ in words in
$M(\Pi)$. (This corresponds to replacing $b$ by $a$ whenever it
appears as non-terminal in a context-free grammar defining $L(\Pi)$.)
Finally, delete $b$ from $A(\Pi)$. Since $a\phi = b\phi$, it follows
that $\Pi$ is a word-hy\-per\-bol\-ic structure admitting an interpretation
$\phi|_{A(\Pi)^+} : A(\Pi)^+ \to S$ with respect to $S$.

Since $A(\Pi)$ is finite, we can iterate this process until
$\phi|_{A(\Pi)}$ becomes injective. Finally, define $\psi =
\phi|_{A(\Pi)^+}$.
\end{proof}

In light of \fullref{Proposition}{prop:wordhypstrinj}, we modify our
definition of interpretation, to insist that each symbol represents a
different element of the semigroup:

\begin{definition}[(Improved version)]
\label{def:improved}
An \defterm{interpretation} of a pre-word-hy\-per\-bol\-ic structure
$\Sigma$ with respect to a semigroup $S$ is a homomorphism $\phi : A(\Sigma)^+
\to S$, with $\phi|_A$ being injective, such that $(L(\Sigma))\phi = S$ and
\[
M(\Sigma) = \gset[\big]{u\#_1v\#_2w^\rev}{u,v,w \in L(\Sigma), (u\phi)(v\phi) = w\phi}.
\]
Again, when there is no risk of confusion, denote $u\phi$ by $\elt{u}$ for
  any $u \in A^+$, and $X\phi$ by $\elt{X}$ for any $X \subseteq A^+$.
\end{definition}

Therefore, a word-hy\-per\-bol\-ic structure is henceforth a
pre-word-hy\-per\-bol\-ic structure that admits an interpretation in the
sense of \fullref{Definition}{def:improved}. With this new definition,
\fullref{Proposition}{prop:isomorphic} shows that each word-hy\-per\-bol\-ic
structure describes a uniquely determined semigroup, and therefore one
can sensibly attempt to solve questions about the semigroup using the
word-hy\-per\-bol\-ic structure.

However, although a word-hy\-per\-bol\-ic structure determines a unique
semigroup, it does not determine a unique interpretation, even up to
automorphic permutation. This parallels the situation for automatic
semigroups \cite[\S~2.2]{kambites_decision}, but is also true in a
rather vacuous sense for word-hy\-per\-bol\-ic semigroups, for the alphabet
$A(\Sigma)$ for a word-hy\-per\-bol\-ic structure $\Sigma$ for a semigroup
$S$ may include a symbol $c$ that does not appear in any word in
either $L(\Sigma)$ or $M(\Sigma)$. (In this situation, $c$ must
represent a redundant generator for $S$.)

For example, let $A(\Sigma) = \set{a,b,c}$, $L(\Sigma) = \set{a,b}^+$, and
$M(\Sigma) = \gset{u\#_1v\#_1(uv)^\rev}{u,v \in \set{a,b}^+}$. Then $\Sigma$
is a word-hy\-per\-bol\-ic structure for the free semigroup $F$ with basis
$\set{x,y}$: let $\phi : A(\Sigma)^+ \to F$ be such that $a\phi = x$ and
$b\phi = y$; regardless of how $c\phi$ is defined, $\phi$ is an
interpretation of $\Sigma$ with respect to $F$.

Less trivial is the following example:

\begin{example}
\label{ex:diffinterp}
Let $S = (\set{1,2} \times
\set{1,2,3}) \cup \set{0_S,1_S}$ and define multiplication on $S$ by
\begin{align*}
(i,\lambda)(j,\mu) &= \begin{cases} 0_S &\text{if $\lambda = j = 1$,} \\
(i,\mu) &\text{otherwise;} \end{cases} \\
1_Sx = x1_S &= x \quad\text{for all $x \in S$;} \\
0_Sx = x0_S &= 0_S \quad\text{for all $x \in S$.}
\end{align*}
Then $S$ is a monoid. (In fact, $S$ is a monoid formed by adjoining an
identity to a $0$-Rees matrix semigroup over the trivial group.)

Let $A(\Sigma) = \set{a,b,c,d,e,i,z}$. Let $L(\Sigma) =
\set{a,b,c,d,ced,dec,i,z}$. Define
\begin{align*}
\phi_1 &: A(\Sigma)^+ \to S & a\phi_1 &= (1,1), & b\phi_1 &= (2,1), & c\phi_1 &= (1,2), \\
&& d\phi_1 &= (2,3), & e\phi_1 &= (2,2), & i\phi_1 &= 1_S, & z\phi_1 &= 0_S; \\
\phi_2 &: A(\Sigma)^+ \to S & a\phi_2 &= (1,1), & b\phi_2 &= (2,1), & c\phi_2 &= (1,2), \\
&& d\phi_2 &= (2,3), & e\phi_2 &= (1,3), & i\phi_1 &= 1_S, & z\phi_2 &= 0_S.
\end{align*}
Notice that the only difference in the definitions of $\phi_1$ and $\phi_2$ is the image of the symbol $e$. Furthermore,
\begin{align*}
(ced)\phi_1 = (1,2)(2,2)(2,3) ={}& (1,3) = (1,2)(1,3)(2,3) = (ced)\phi_2,\\
(dec)\phi_1 = (2,3)(2,2)(1,2) ={}& (2,2) = (2,3)(1,3)(1,2) = (dec)\phi_2;
\end{align*}
thus $\phi_1|_L = \phi_2|_L$ and $L\phi_1 = L\phi_2 = S$.

Define
\[
M(\Sigma) = \gset[\big]{u\#_1v\#_2w^\rev}{u,v,w \in L(\Sigma), (u\phi_1)(v\phi_1) = w\phi_1}.
\]
Since $L(\Sigma)$ is finite, $M(\Sigma)$ is also finite and thus context-free. So $\Sigma$ is a word-hy\-per\-bol\-ic
structure for $S$ and $\phi_1$ is an interpretation for $\Sigma$ with respect to $S$.  Furthermore, since Hence
$\phi_1|_L = \phi_2|_L$,
\[
M(\Sigma) = \gset[\big]{u\#_1v\#_2w^\rev}{u,v,w \in L(\Sigma), (u\phi_2)(v\phi_2) = w\phi_2},
\]
and so $\phi_2$ is also an interpretation of $\Sigma$ with respect to
$S$.

Moreover, there is no automorphism $\rho$ of $S$ such that $\phi_1\rho = \phi_2$. To see this, notice that such a $\rho$
would have to map $e\phi_1 = (2,1)$ to $e\phi_2 = (1,3)$. The map $\rho$ would also preserve $\gR$-classes. But the
$\gR$-class of $(1,3)$ contains the element $(1,1)$, which is not idempotent (since $(1,1)(1,1) = 0$), whereas every
element of the $\gR$-class of $(2,1)$ is idempotent. So no such map $\rho$ can exist. So the two interpretations are not
even equivalent up to automorphic permutation of $S$.
\end{example}

The crucial point in \fullref{Example}{ex:diffinterp} is that
\fullref{Proposition}{prop:isomorphic} only guarantees that the
\emph{restriction} of two interpretations to $L$ are equivalent up to
automorphic permutation. It says nothing about the interpretation maps
on the whole of $A^+$.

The next result essentially shows that word-hyperbolicity is invariant under change of finite generating set. This
result, and its proof, are due to Hoffmann et al. \cite[Proposition~4.2]{hoffmann_relatives}, but are given here using
the more precise definitions of the present paper.

\begin{proposition}
  \label{prop:wordhypchangegens}
  Let $S$ be a word-hyperbolic semigroup. Let $X \subseteq S$ be a finite generating set for $S$. Then there is a
  word-hyperbolic structure $\Sigma$ for $S$ with an interpretation $\phi : A(\Sigma)^+ \to S$ such that
  $(A(\Sigma))\phi = X$.
\end{proposition}

\begin{proof}
  Let $\Pi$ be a word-hyperbolic structure for $S$. Let $\psi : A(\Pi)^+ \to S$ be an interpretation for $S$. Let
  $A(\Sigma)$ be an alphabet in bijection with $X$ under a map $\phi : A(\Sigma) \to X$. The map $\phi$ extends to a
  unique homomorphism $\phi : A(\Sigma) \to S$, which is surjective since $X$ generates $S$. Note that
  $\phi|_{A(\Sigma)}$ is injective by construction.

  For each $a \in A(\Pi)$, let $w_a$ be a word in $A(\Sigma)^+$ such that $a\psi = w_a\phi$; such words exist since
  $\phi$ is surjective. Define a homomorphism
  \begin{align*}
    \Theta : A(\Pi)^+ \cup \set{\#_1,\#_2} \to A(\Sigma)^+ &&a \mapsto w_a, &&\text{for $a \in A(\Pi)$} \\
    &&\#_i \mapsto \#_i &&\text{for $i=1,2$.}
  \end{align*}
  Let $L(\Sigma) = (L(\Pi))\Theta$ and $M(\Sigma) = (M(\Pi))\Theta$. The class of regular languages is closed under
  homomorphism, so $L(\Sigma)$ is regular; similarly, $M(\Sigma)$ is context-free. So $\Sigma$ is a pre-word-hyperbolic
  structure; it remains to prove that $\phi$ is an interpretation of $\Sigma$.

  Notice that $a\psi = w_a\phi = a\Theta\phi$, and thus $\psi = \Theta\phi$. Therefore $(L(\Sigma))\phi =
  (L(\Pi))\Theta\phi = (L(\Pi))\psi = S$, since $\psi$ is an interpretation of $\Pi$. Furthermore,
  \begin{align*}
    &u\#_1v\#_2w^\rev \in M(\Sigma) \\
    \iff{}& (\exists u',v',w' \in L(\Pi))(u'\Theta = u \land v'\Theta = v \land w'\Theta = w \land u'\#_1v'\#(w')^\rev \in M(\Pi)) \\
    \iff{}& (\exists u',v',w' \in L(\Pi))(u'\Theta = u \land v'\Theta = v \land w'\Theta = w \land (u'\psi)(v'\psi) = w'\psi) \\
    \iff{}& (\exists u',v',w' \in L(\Pi))(u'\Theta = u \land v'\Theta = v \land w'\Theta = w \land (u'\Theta\phi)(v'\Theta\phi) = w'\Theta\phi) \\
    \iff{}& u,v,w \in L(\Sigma) \land (u\phi)(v\phi) = w\phi).
  \end{align*}
  Thus $\phi$ is an interpretation of $\Sigma$.
\end{proof}

In order to compute with the semigroup described by a word-hy\-per\-bol\-ic
structure, interpretations must be coded in a finite way.

\begin{definition}
An \defterm{assignment of generators} for a word-hy\-per\-bol\-ic structure
$\Sigma$ is a map $\alpha : A(\Sigma) \to L(\Sigma)$ with the property
that there is some interpretation $\phi : A(\Sigma)^+ \to S$ such that
$a\alpha\phi = a\phi$ for all $a \in A$; such an interpretation is
said to be \defterm{consistent} with $\alpha$. Two assignments of
generators $\alpha$ and $\beta$ for $\Sigma$ are \defterm{equivalent}
if $(a\alpha,a\beta) \in E(\Sigma)$ for all $a \in A(\Sigma)$.
\end{definition}

\begin{proposition}
An assignment of generators for a word-hy\-per\-bol\-ic structure is
consistent with a unique interpretation (up to automorphic permutation
of the semigroup described). Equivalent assignments of generators are
consistent with the same interpretation.

Conversely, every interpretation is consistent with a unique (up to
equivalence) assignment of generators.
\end{proposition}

\begin{proof}
Let $\Sigma$ be a word-hy\-per\-bol\-ic structure and $\alpha : A \to L$ an
assignment of generators. Then there is an interpretation $\phi : A^+
\to S$ of $\Sigma$ that is consistent with $\alpha$; that is,
$a\alpha\phi = a\phi$ for all $a \in A$.

Let $\psi : A^+ \to S$ be another interpretation of $\Sigma$ that is
consistent with $\alpha$; the aim is to show that $\phi$ and $\phi$
differ only by an automorphic permutation of $S$. First, $a\alpha\psi
= a\psi$ for all $a \in A$, since $\psi$ is consistent with
$\alpha$. By \fullref{Proposition}{prop:isomorphic}, there is an
automorphism $\tau$ of $S$ such that $\phi|_L\tau = \psi|_L$, and so
$a\phi\tau = a\alpha\phi\tau = a\alpha\psi = a\psi$ for all $a \in
A$. Hence $\phi$ and $\psi$ differ only by the automorphism $\tau$.

Now let $\beta : A \to L$ be an assignment of generators equivalent to
$\alpha$; the aim is to show that $\beta$ is also consistent with
$\phi$. Now, $(a\alpha,a\beta) \in E(\Sigma)$ for all $a \in A$ since
$\alpha$ and $\beta$ are equivalent. Thus $a\phi = a\alpha\phi =
a\beta\phi$ by \fullref{Lemma}{lem:equality}, and hence $\beta$ is
also consistent with the interpretation $\phi$.

Finally, suppose $\gamma : A \to L$ is an assignment of generators
consistent with $\phi$; the aim is to show $\alpha$ and $\beta$ are
equivalent. Now, $a\alpha\phi = a\phi = a\gamma\phi$ for all $a \in A$
since $\alpha$ and $\gamma$ are consistent with $\phi$. Hence
$(a\alpha,a\gamma) \in E(\Sigma)$ for all $a \in A$ by
\fullref{Lemma}{lem:equality}.
\end{proof}

\begin{definition}
A word-hy\-per\-bol\-ic structure $\Sigma$ is said to be an
\defterm{interpreted} word-hy\-per\-bol\-ic structure if it is equipped with
an assigment of generators $\alpha(\Sigma)$.
\end{definition}

\begin{proposition}
\label{prop:gensarewords}
Let $\Sigma$ be an interpreted word-hy\-per\-bol\-ic structure for a
semigroup $S$. Then there is another interpreted word-hy\-per\-bol\-ic
structure $\Sigma'$ for $S$, effectively computable from $\Sigma$,
such that $A(\Sigma') \subseteq L(\Sigma')$ and $\alpha(\Sigma')$ is
the embedding map from $A(\Sigma')$ to $L(\Sigma')$.
\end{proposition}

\begin{proof}
Let $A(\Sigma')$ be $A(\Sigma)$ and $L(\Sigma') = L(\Sigma) \cup
A(\Sigma)$. For brevity, let $\alpha = \alpha(\Sigma)$. For each
$a,b,c \in A(\Sigma')$, define the languages:
\begin{align*}
M^{(1)}_a &= \gset{a\#_1v\#_2w^\rev }{(a\alpha)\#_1v\#_2w^\rev \in M(\Sigma)}, \\
M^{(2)}_a &= \gset{u\#_1a\#_2w^\rev }{u\#_1(a\alpha)\#_2w^\rev \in M(\Sigma)}, \displaybreak[0]\\
M^{(3)}_a &= \gset{u\#_1v\#_2a^\rev}{u\#_1v\#_2(a\alpha)^\rev \in M(\Sigma)}, \displaybreak[0]\\
M^{(4)}_{a,b} &= \gset{a\#_1b\#_2w^\rev}{(a\alpha)\#_1(b\alpha)\#_2w^\rev \in M(\Sigma)}, \displaybreak[0]\\
M^{(5)}_{a,b} &= \gset{u\#_1a\#_2b^\rev}{u\#_1(a\alpha)\#_2(b\alpha)^\rev \in M(\Sigma)}, \displaybreak[0]\\
M^{(6)}_{a,b} &= \gset{a\#_1v\#_2b^\rev}{(a\alpha)\#_1v\#_2(b\alpha)^\rev \in M(\Sigma)}, \\
M^{(7)}_{a,b,c} &= \gset{a\#_1b\#_2c^\rev}{(a\alpha)\#_1(b\alpha)\#_2(c\alpha)^\rev \in M(\Sigma)}.
\end{align*}
Each of these languages is context-free because each is the
intersection of the context-free language $M(\Sigma)$ with a regular
language. [Notice that $M^{(7)}_{a,b,c}$ is either empty or a
  singleton language.]

Now let
\begin{align*}
M(\Sigma') = M(\Sigma) &\cup \bigcup_{a \in A(\Sigma)} \bigl(M^{(1)}_a \cup M^{(2)}_b \cup M^{(3)}_c\bigr) \\
&\cup \bigcup_{a,b \in A(\Sigma)} \bigl(M^{(4)}_{a,b} \cup M^{(5)}_{a,b} \cup M^{(6)}_{a,b}\bigr) \\
&\cup \bigcup_{a,b,c \in A(\Sigma)} M^{(7)}_{a,b,c};
\end{align*}
notice that $M(\Sigma')$ is also context-free.

Let $\phi : A(\Sigma) \to S$ be an interpretation of $\Sigma$. Then,
recalling that $A(\Sigma) = A(\Sigma')$,
\[
M(\Sigma') = \gset{u\#_1v\#_2w^\rev}{u,v,w \in L(\Sigma'), (u\phi)(v\phi) = w\phi},
\]
because $u$, $v$, and $w$ range over $L(\Sigma') = L(\Sigma) \cup
A(\Sigma)$, and the eight cases that arise depending on whether each
word lies in $L(\Sigma)$ or $A(\Sigma)$ correspond to the eight sets
$M(\Sigma)$, $M^{(1)}_a$, $M^{(2)}_a$, $M^{(3)}_a$, $M^{(4)}_{a,b}$,
$M^{(5)}_{a,b}$, $M^{(6)}_{a,b}$, and $M^{(7)}_{a,b,c}$. [Notice that
  these sets are not necessarily disjoint, since it is possible that
  $a\alpha = a$ for some $a \in A$.]

Finally, define $\alpha(\Sigma')$ to be the embedding map from
$A(\Sigma')$ to $L(\Sigma')$. This map is an assignment of generators
since trivially $((a)\alpha(\Sigma'))\phi = a\phi$ for any
interpretation $\phi$ of $\Sigma'$.
\end{proof}

In light of \fullref{Proposition}{prop:gensarewords}, we will assume
without further comment that an interpreted word-hy\-per\-bol\-ic structure
$\Sigma$ has the property that $A(\Sigma) \subseteq L(\Sigma)$ and
that $\alpha(\Sigma)$ is the embedding map from $A(\Sigma)$ to
$L(\Sigma)$. Notice further that the computational effectiveness
aspect of \fullref{Proposition}{prop:gensarewords} ensures we are free
to assume that an interpreted word-hy\-per\-bol\-ic structure serving as
input to a decision problem has this property.

For automatic semigroups, it is possible to assume that the automatic structure has a further pleasant property, namely
that every element of the semigroup is represented by a unique word in the language of representatives
\cite[Proposition~2.9(iii)]{kambites_decision}. However, there exist word-hy\-per\-bol\-ic semigroups (indeed,
word-hy\-per\-bol\-ic monoids) that do not admit word-hy\-per\-bol\-ic structures where the languages of representatives
have this uniqueness property \cite[Examples~10 \&~11]{cm_wordhypunique}.

\section{Other background }

\subsection{String-rewriting systems}

This section recalls the necessary basic definitions and terminology on string-rewriting systems and their connection to
semigroup presentations; \cite{book_srs}, and \cite{baader_termrewriting} for further background reading.

A \defterm{string rewriting system}, or simply a \defterm{rewriting
  system}, is a pair $(A,\drel{R})$, where $A$ is a finite alphabet and
$\drel{R}$ is a set of pairs $(\ell,r)$, usually written $\ell
\imreduces r$, known as \defterm{rewriting rules} or simply
\defterm{rules}, drawn from $A^* \times A^*$. The single reduction
relation $\imreduces_{\drel{R}}$ is defined as follows: $u
\imreduces_{\drel{R}} v$ (where $u,v \in A^*$) if there exists a
rewriting rule $(\ell,r) \in \drel{R}$ and words $x,y \in A^*$ such
that $u = x\ell y$ and $v = xry$. That is, $u \imreduces_{\drel{R}} v$
if one can obtain $v$ from $u$ by substituting the word $r$ for a
subword $\ell$ of $u$, where $\ell \imreduces r$ is a rewriting
rule. The reduction relation $\reduces_{\drel{R}}$ is the reflexive and
transitive closure of $\imreduces_{\drel{R}}$. The subscript $\drel{R}$
is omitted when it is clear from context. The process of replacing a
subword $\ell$ by a word $r$, where $\ell \imreduces r$ is a rule, is
called \defterm{reduction} by application of the rule $\ell \imreduces
r$; the iteration of this process is also called reduction. A word $w
\in A^*$ is \defterm{reducible} if it contains a subword $\ell$ that
forms the left-hand side of a rewriting rule in $\drel{R}$; it is
otherwise called \defterm{irreducible}.

The rewriting system $(A,\drel{R})$ is \defterm{finite} if both $A$ and $\drel{R}$ are finite. The rewriting system
$(A,\drel{R})$ is \defterm{noetherian} if there is no infinite sequence $u_1,u_2,\ldots$ of words from $A^*$ such that
$u_i \imreduces u_{i+1}$ for all $i \in \nset$. That is, $(A,\drel{R})$ is noetherian if any process of reduction must
eventually terminate with an irreducible word. The rewriting system $(A,\drel{R})$ is \defterm{confluent} if, for any
words $u, u',u'' \in A^*$ with $u \reduces u'$ and $u \reduces u''$, there exists a word $v \in A^*$ such that $u'
\reduces v$ and $u'' \reduces v$.  It is well known that a noetherian system is confluent if and only if all critical
pairs resolve, where critical pairs are obtained by considering overlaps of left hand sides of the rewrite rules
$\drel{R}$; see \cite{book_srs} for more details.  A rewriting system that is both confluent and noetherian is
\defterm{complete}.

The rewriting system $(A,\drel{R})$ is \defterm{length-reducing} if $(\ell,r) \in \drel{R}$ implies that $|\ell| > |r|$.
Observe that any length-reducing rewriting system is necessarily Noetherian. The rewriting system $(A,\drel{R})$ is
\defterm{monadic} if it is length-reducing and the right-hand side of each rewrite rule in $\drel{R}$ lies in
$A \cup \{\emptyword\}$. A monadic rewriting system $(A,\drel{R})$ is \defterm{regular} (respectively,
\defterm{context-free}) if, for each $a \in A \cup\{\emptyword\}$, the set of all left-hand sides of rewrite rules in
$\drel{R}$ with right-hand side $a$ is a regular (respectively, context-free) language.

The \defterm{Thue congruence} $\thue_{\drel{R}}$ is the equivalence
relation generated by $\imreduces_{\drel{R}}$. The elements of the monoid
presented by $\pres{A}{\drel{R}}$ are the $\thue_{\drel{R}}$-equivalence
classes. The relations $\thue_{\drel{R}}$ and $\cgen{\drel{R}}$ coincide.

\subsection{Automaticity}

This subsection contains the definitions and basic results from the theory of automatic and biautomatic monoids needed
hereafter. For further information on automatic semigroups, see~\cite{campbell_autsg}.

\begin{definition}
Let $A$ be an alphabet and let $\$$ be a new symbol not in $A$. Define
the mapping $\rpad : A^* \times A^* \to ((A\cup\set{\$})\times (A\cup
\set{\$}))^*$ by
\[
(u_1\cdots u_m,v_1\cdots v_n) \mapsto
\begin{cases}
(u_1,v_1)\cdots(u_m,v_n) & \text{if }m=n,\\
(u_1,v_1)\cdots(u_n,v_n)(u_{n+1},\$)\cdots(u_m,\$) & \text{if }m>n,\\
(u_1,v_1)\cdots(u_m,v_m)(\$,v_{m+1})\cdots(\$,v_n) & \text{if }m<n,
\end{cases}
\]
and the mapping $\lpad : A^* \times A^* \to ((A\cup\set{\$})\times (A\cup \set{\$}))^*$ by
\[
(u_1\cdots u_m,v_1\cdots v_n) \mapsto
\begin{cases}
(u_1,v_1)\cdots(u_m,v_n) & \text{if }m=n,\\
(u_1,\$)\cdots(u_{m-n},\$)(u_{m-n+1},v_1)\cdots(u_m,v_n) & \text{if }m>n,\\
(\$,v_1)\cdots(\$,v_{n-m})(u_1,v_{n-m+1})\cdots(u_m,v_n) & \text{if }m<n,
\end{cases}
\]
where $u_i,v_i \in A$.
\end{definition}

\begin{definition}
\label{def:autstruct}
Let $M$ be a monoid. Let $A$ be a finite alphabet representing a set
of generators for $M$ and let $L \subseteq A^*$ be a regular language such
that every element of $M$ has at least one representative in $L$.  For
each $a \in A \cup \set{\emptyword}$, define the relations
\begin{align*}
L_a &= \gset{(u,v)}{u,v \in L, \elt{ua} = \elt{v}}\\
{}_aL &= \gset{(u,v)}{u,v \in L, \elt{au} = \elt{v}}.
\end{align*}
The pair $(A,L)$ is an \defterm{automatic structure} for $M$ if $L_a\rpad$ is a regular language over $(A\cup\set{\$})
\times (A\cup\set{\$})$ for all $a \in A \cup \set{\emptyword}$. A monoid $M$ is \defterm{automatic} if it admits an
automatic structure with respect to some generating set.

The pair $(A,L)$ is an \defterm{asynchronous automatic structure} for $M$ if $L_a$ is a rational relation for all $a \in
A \cup \set{\emptyword}$. A monoid $M$ is \defterm{asychronously automatic} if it admits an asynchronous automatic
structure with respect to some generating set.

The pair $(A,L)$ is a \defterm{biautomatic structure} for $M$ if $L_a\rpad$, ${}_aL\rpad$, $L_a\lpad$, and ${}_aL\lpad$
are regular languages over $(A\cup\set{\$}) \times (A\cup\set{\$})$ for all $a \in A \cup \set{\emptyword}$. A monoid $M$ is
\defterm{biautomatic} if it admits a biautomatic structure with respect to some generating set.

The pair $(A,L)$ is an \defterm{asynchronous biautomatic structure} for $M$ if $L_a$, ${}_aL$, $L_a$, and ${}_aL$ are
rational languages for all $a \in A \cup \set{\emptyword}$. A monoid $M$ is \defterm{asynchronouly biautomatic} if it
admits an asynchronous biautomatic structure with respect to some generating set.
\end{definition}

Unlike in the situation for groups, biautomaticity and automaticity for semigroups is dependent on the choice of generating
set \cite[Example~4.5]{campbell_autsg}. However, for monoids, biautomaticity and automaticity are independent of the
choice of \emph{semigroup} generating sets \cite[Theorem~1.1]{duncan_change}.

Note that biautomaticity implies automaticity and asynchronous biautomacitity, and both of these properties imply
asynchronous automaticity.

Hoffmann \& Thomas have made a careful study of biautomaticity for semigroups \cite{hoffmann_biautomatic}. They
distinguish four notions of biautomaticity for semigroups, which are all equivalent for groups and more generally for
cancellative semigroups \cite[Theorem~1]{hoffmann_biautomatic} but distinct for semigroups \cite[Remark~1 \&
\S~4]{hoffmann_biautomatic}. In the sense used in this paper, `biautomaticity' implies \emph{all four} of these notions
of biautomaticity.

Although automaticity for semigroups (unlike for groups) is dependent on the choice of generating set
\cite[Example~4.5]{campbell_autsg}, asynchronous automaticity is maintained under change of generators:

\begin{proposition}[{\cite[Proposition~4.1]{hoffmann_relatives}}]
  \label{prop:asynchautochangegen}
  Let $S$ be an asynchronously automatic semigroup, and let $A$ be a finite alphabet representing a generating set for
  $S$. Then there is a regular language $L$ over $A$ such that $(A,L)$ is an asynchronous automatic structure for $S$.
\end{proposition}

\subsection{Normal forms of context-free grammars}

This subsection recalls some slightly less well-known terminology about normal forms for context-free grammars, and
proves a technical lemma.

Consider a context-free grammar $\Gamma = (N,A,P,S)$ (where $N$ is the non-terminal alphabet, $A$ the terminal alphabet,
$P$ the set of productions, and $S$ the start symbol). The \defterm{size} of $\Gamma$, denoted $|\Gamma|$, is the sum of
the lengths of the right-hand sides of the productions in $P$; that is,
$|\Gamma| = \sum \gset{|\alpha|}{X \to \alpha \in P}$. A \defterm{unit production} (or \defterm{chain rule}) is a
production of the form $X \to Y$ for $X,Y \in N$. The grammar $\Gamma$ is $\emptyword$-free if every production is of
one of the two forms
\[
X \to \alpha \text{ with $\alpha \in (A \cup N - \set{S})^+$}; \qquad S \to \emptyword.
\]
The grammar $\Gamma$ is in \defterm{Chomsky normal form} if every production is of
one of the three forms
\begin{equation}
\label{eq:cnfprods}
X \to YZ \text{ with $Y,Z \in N - \set{S}$}; \qquad X \to a \text{ with $a \in A$}; \qquad S \to \emptyword.
\end{equation}
It is in \defterm{extended Chomsky normal form} if every production is either a unit production or of one of the forms
\eqref{eq:cnfprods}.
The grammar $\Gamma$ is in \defterm{Greibach normal form} if every
production is of one of the two forms
\begin{equation}
\label{eq:gnfprods}
X \to a\alpha \text{ with $a \in A$ and $\alpha \in (N - \set{S})^+$}; \qquad S \to \emptyword.
\end{equation}
It is in \defterm{extended Greibach normal form} if every production is either a unit production or of one of the forms
\eqref{eq:gnfprods}. The grammar $\Gamma$ is in \defterm{quadratic (extended) Greibach normal form} (or \defterm{$2$
  (extended) Greibach normal form}) if every production of the form $X \to a\alpha$ satisfies $|\alpha| \leq 2$.

\begin{lemma}
  \label{lem:conversion}
  There is an algorithm that takes as input an $\emptyword$-free context-free grammar $\Gamma$ and outputs a quadratic
  Greibach normal form grammar $\Gamma_{\mathrm{G}}$, taking time $\bigO(|\Gamma|^2)$.
\end{lemma}

\begin{proof}
  The strategy is to follow the construction used by Blum \& Koch \cite[Paragraph following Theorem~2.1]{blum_greibach}
  and note the time complexity at each stage.

  The first step is to convert $\Gamma$ to an extended Chomsky normal form grammar $\Gamma_{\mathrm{EC}}$; this takes
  time $\bigO(|\Sigma|)$ by inspection of the usual construction (see, for example, \cite[Proof of
  Theorem~4.5]{hopcroft_automata}, ignoring the removal of unit productions), and $|\Gamma_{\mathrm{EC}}|$ is at most a
  constant multiple of $|\Gamma|$.

  The next step is Blum \& Koch's own construction \cite[p.116]{blum_greibach} to convert $\Gamma_{\mathrm{EC}}$ to an
  quadratic Greibach normal form grammar $\Gamma_{\mathrm{eg}}$. This involves first constructing auxiliary grammars
  $\Gamma_X$ for all $X$ in $N - \set{S}$; by inspection this takes time $\bigO(|\Gamma_{\mathrm{EC}}|)$ for each $X$,
  and thus $\bigO(|\Gamma_{\mathrm{EC}}|^2)$ time in total, and the grammars $\Gamma_X$ have size at most
  $3|\Gamma_{\mathrm{EC}}|$. The final construction of the quadratic Greibach normal form grammar $\Gamma_{\mathrm{G}}$
  from $\Gamma_{\mathrm{EC}}$ and the various $\Gamma_X$ thus takes time $\bigO(|\Gamma_{\mathrm{EC}}|^2)$.

  Since $|\Gamma_{\mathrm{EC}}|$ is at most a constant multiple of $|\Gamma|$, the construction of $\Gamma_{\mathrm{G}}$
  takes time $\bigO(|\Gamma|^2)$.
\end{proof}

\section{Isomorphism problem \& automaticity}
\label{sec:isoundec}

This section proves that the isomorphism problem, and the problem of deciding automaticity, are both undecidable for
word-hyperbolic semigroups. [Recall that, as noted in the introduction, a word-hyperbolic monoid is not necessarily
automatic or asynchronously automatic \cite[Example~7.7 et~seq.]{hoffmann_relatives}.]  The strategy is to reduce
standard undecidable questions for context-free grammars to these decision problem for word-hyperbolic semigroups; this
will show that these problems are also undecidable. The two undecidable questions we will use are the problems of
deciding, from a given context-free grammar, whether it defines the language of all words over an alphabet, and whether
it defines a regular language:

\begin{theorem}[{\cite[Theorem~8.11]{hopcroft_automata}}]
  \label{thm:cfgallwordsundec}
  There is no algorithm that takes as input a context-free grammar $\Gamma$ over a finite alphabet $B$ and decides whether
  $L(\Gamma) = B^*$.
\end{theorem}

\begin{theorem}[{\cite[Theorem~8.15]{hopcroft_automata}}]
  \label{thm:cfgregularundec}
  There is no algorithm that takes as input a context-free grammar $\Gamma$ over a finite alphabet $B$ and decides whether
  $L(\Gamma)$ is regular.
\end{theorem}

The key to encoding these undecidability results into decision problems for word-hyperbolic semigroups is the
following result, due to the first author and Maltcev:

\begin{theorem}[{\cite[Theorem~3.1]{cm_wordhypunique}}]
  \label{thm:contextfreesrs}
  Let $(A,\drel{R})$ be a confluent context-free monadic rewriting system where $\drel{R}$ does not contain rewriting
  rules with $\emptyword$ on the right-hand side. Then there is an intepreted word-hyperbolic structure $\Sigma$ for the
  semigroup presented by $\pres{A}{\drel{R}}$ such that $A(\Sigma) = A$ and $L(\Sigma) = A^*$. Furthermore, $\Sigma$ can
  be effectively constructed from context-free grammars describing $\drel{R}$.
\end{theorem}

(The preceding result was originally stated for monoids, allowing $\drel{R}$ to contain rules with $\emptyword$ on the
right-hand side; it is immediate that it holds in this form for semigroups. The `effective construction' part follows
easily by inspecting the construction of the word-hyperbolic structure in the proof.)

\begin{lemma}
  \label{lem:cfgundecconst}
  Let $\Gamma$ be a context-free grammar over a finite alphabet $B$. Let $x,y,z$ be new symbols not in $B$ and let $A =
  B \cup \set{x,y,z}$. Define $\drel{R}_1$, $\drel{R}_2$, and $\drel{Z}$ as follows:
  \begin{align*}
    \drel{R}_1 &= \gset[\big]{xwy \to z}{w \in L(\Gamma)}, \\
    \drel{R}_2 &= \gset[\big]{xwy \to z}{w \in B^*}, \\
    \drel{Z} &= \gset[\big]{az \to z,za \to z}{a \in A}.
  \end{align*}
  Let $S_1$ and $S_2$ be the semigroups presented by, respectively, $\pres{A}{\drel{R}_1 \cup \drel{Z}}$ and
  $\pres{A}{\drel{R}_2 \cup \drel{Z}}$. Then:
  \begin{enumerate}
  \item $S_1$ and $S_2$ are word-hyperbolic, with effectively computable word-hyperbolic structures.
  \item $L(\Gamma) = B^*$ if and only if $S_1$ and $S_2$ are isomorphic.
  \item The following are equivalent:
    \begin{enumerate}
    \item $L(\Gamma)$ is regular;
    \item $S_1$ is biautomatic;
    \item $S_1$ is asynchronously biautomatic;
    \item $S_1$ is automatic;
    \item $S_1$ is asynchronously automatic.
    \end{enumerate}
  \end{enumerate}
\end{lemma}

\begin{proof}
  \begin{enumerate}
  \item Notice first that $(A,\drel{R}_1 \cup \drel{Z})$ is a context-free monadic rewriting system. It is confluent
    because any rewriting must produce a symbol $z$, and so the entire word rewrites to $z$ using rewriting rules in
    $\drel{Z}$. Similarly, $(A,\drel{R}_2 \cup \drel{Z})$ is a confluent context-free monadic rewriting system. By
    \fullref{Theorem}{thm:contextfreesrs}, $S_1$ and $S_2$ have effectively computable word-hyperbolic structures
    $\Sigma_1$ and $\Sigma_2$ such that $A(\Sigma_1) = A$ and $A(\Sigma_2) = A$.  Let $\phi_1 : A(\Sigma_1)^+ \to S_1$
    and $\phi_2 : A(\Sigma_2)^+ \to S_2$ be interpretations of $\Sigma_1$ and $\Sigma_2$.

  \item Suppose that $L(\Gamma) = B^*$. Then $\drel{R}_1 = \drel{R}_2$ and so $S_1$ and $S_2$ are isomorphic.

    Now suppose $S_1$ and $S_2$ are isomorphic. Let $\tau : S_1 \to S_2$ be an isomorphism. Now, $z\phi_1$ and $z\phi_2$
    are the unique zeroes of $S_1$ of $S_2$, so $\tau$ must map $z\phi_1$ to $z\phi_2$. Furthermore, $x\phi_1$ and
    $x\phi_2$ are the unique non-trivial left divisors of $z\phi_1$ and $z\phi_2$, respectively. Hence $\tau$ maps
    $x\phi_1$ to $x\phi_2$. Similarly, $\tau$ maps $y\phi_1$ and $y\phi_2$. Since all elements of $B\phi_1$ and
    $B\phi_2$ are indecomposable, $\tau$ must map $B\phi_1$ to $B\phi_1$ and thus $\tau$ restricts to an isomorphism
    between the free subsemigroups $B^+\phi_1$ and $B^+\phi_2$.

    Suppose, with the aim of obtaining a contradiction, that $L(\Gamma) \subsetneq B^*$. Let
    $u \in B^* \setminus L(\Gamma)$ and let $v \in B^*$ be such that $u\phi_1\tau = v\phi_2$. Then
    $(x\phi_1)(u\phi_1)(y\phi_1) = (xuy)\phi_1 \neq z\phi_1$ since no rewriting rule in $\drel{R}_1 \cup \drel{Z}$ can
    be applied to $xuy$. But $(x\phi_1\tau)(u\phi_1\tau)(y\phi_2\tau) = (xvy)\phi_2 = z\phi_2 = z\phi_1\tau$ by the
    rules in $\drel{R}_2 \cup \drel{Z}$, which contradicts $\tau$ being an isomorphism. Hence $L(\Gamma) = B^*$.

    Thus $L(\Gamma) = B^*$ if and only if $S_1$ and $S_2$ are isomorphic.

  \item Suppose $L(\Gamma)$ is regular. Then $(A,\drel{R}_1 \cup \drel{Z})$ is a regular monadic rewriting system, and
    is confluent by the reasoning in the proof of part~1. Let $L$ be the language of normal forms for
    $(A,\drel{R}_1 \cup \drel{Z})$; it is easy to see that
    \[
    L = (A - \set{z})^+ - A^*xL(\Gamma)yA^* \cup \set{z},
    \]
    and that
    \begin{align*}
      L_\emptyword = {}_\emptyword L &= \gset{(u,u)}{u \in L} \\
      L_a &= \gset{(u,ua)}{u \in L} \qquad\qquad\text{for all $a \in A - \set{y,z}$} \\
      L_y &= \gset{(u,uy)}{u \in L - A^*xL(\Gamma)} \cup \gset{(u,z)}{u \in A^*xL(\Gamma)} \\
      L_z &= \gset{(u,z)}{u \in L}; \\
      {}_aL &= \gset{(au,u)}{u \in L} \qquad\qquad\text{for all $a \in A - \set{x,z}$} \\
      {}_xL &= \gset{(xu,u)}{u \in L - L(\Gamma)yA^*} \cup \gset{(u,z)}{u \in L(\Gamma)yA^*} \\
      {}_zL &= \gset{(u,z)}{u \in L}.
    \end{align*}
    It is easy to see that $L_a\rpad$, $L_a\lpad$, ${}_aL\lpad$ and ${}_aL\rpad$ are all regular for all $a \in A \cup \set{\emptyword}$. Thus $(A,L)$ is a
    biautomatic structure for $S_1$. Thus a) implies b).

    It is clear that b) implies c) and d), and both c) and d) imply e).

    So suppose $S_1$ is asynchronously automatic. By \fullref{Proposition}{prop:asynchautochangegen}, $S_1$ admits an
    asynchronous automatic structure $(A,L)$. If $w \in B^+ - L(\Gamma)$, then $xwy$ is the unique word over $A$
    representing $\elt{xwy} \in S_1$ (since no relation in $\drel{R}_1 \cup \drel{Z}$ can be applied to $xwy$). Hence
    $x(B^+ - L(\Gamma))y \subseteq L$. Note also that the language $K$ of words in $L$ representing $\elt{z}$ is regular, since
    \[
    K = \gset{u \in L}{(u,z) \in L_{\emptyword}}.
    \]
    Since $xwy$ represents $\elt{z}$ if and only if $w \in L(\Gamma)$, it follows that $x(B^+ - L(\Gamma))y = xB^+y -
    K$. So $x(B^+ - L(\Gamma))y$ is regular, and thus $L(\Gamma) = B^+ - x^{-1}(x(B^+ - L(\Gamma)y)y^{-1}$ is
    regular. Thus e) implies a). \qedhere
  \end{enumerate}
\end{proof}

The two undecidability results can now be deduced from the preceding lemma:

\begin{theorem}
  \label{thm:isoundec}
  The isomorphism problem is undecidable for word-hyperbolic semigroups. That is, there is no algorithm that takes as
  input two interpreted word-hyperbolic structures $\Sigma_1$ and $\Sigma_2$ for semigroups $S_1$ and $S_2$ and decides
  whether $S_1$ and $S_2$ are isomorphic.
\end{theorem}

\begin{proof}
  Since there is no algorithm that takes a context-free grammar $\Gamma$ and decides whether $L(\Gamma) = B^*$ by
  \fullref{Theorem}{thm:cfgallwordsundec}, it follows from \fullref[(1,2)]{Lemma}{lem:cfgundecconst} that there is no
  algorithm that takes two interpreted word-hyperbolic structures and decides whether the semigroups they define are
  isomorphic.
\end{proof}

\begin{theorem}
  \label{thm:autoundec}
  It is undecidable whether a word-hyperbolic semigroup is automatic (respectively, asynchronously automatic,
  biautomatic, asynchronously biautomatic). That is, there is no algorithm that takes as input an interpreted
  word-hyperbolic structure for a semigroup $S$ decides whether $S$ is automatic (respectively,
  asynchronously automatic, biautomatic, asynchronously biautomatic).
\end{theorem}

\begin{proof}
  Since there is no algorithm that takes a context-free grammar $\Gamma$ and decides whether $L(\Gamma)$ is regular
  \fullref{THeorem}{thm:cfgregularundec}, it follows from \fullref[(1,3)]{Lemma}{lem:cfgundecconst} that there is no
  algorithm that takes as input an interpreted word-hyperbolic structure and decides whether the semigroup it defines is
  automatic (respectively, asynchronously automatic, biautomatic, asynchronously biautomatic).
\end{proof}

\section{Basic calculations}

This section notes a few very basic facts about computing with word-hy\-per\-bol\-ic structures for semigroups that are
used later in the paper.

\begin{lemma}[{\cite[Lemma~3.6 \& its proof]{hoffmann_relatives}}]
  \label{lem:multiplyingqgnf}
  There is an algorithm that takes as input a word-hyperbolic structure $\Sigma$ for a semigroup, with $M(L)$ being
  specified by a context-free grammar in quadratic Greibach normal form, and two words $p,q \in L(\Sigma)$, and outputs
  a word $r \in L(\Sigma)$ satisfying $\elt{p}\,\elt{q} = \elt{r}$ with $|r| \leq c(|p| + |q|)$ (where $c$ is a constant
  dependent only on $\Sigma$) in time $\bigO((|p|+|q|)^5)$.
\end{lemma}

(Actually, the appearance of this lemma in \cite{hoffmann_relatives} allows $p$ or $q$ to be empty and asserts that
$|r| \leq c(|p| + |q| + 2)$. To obtain the lemma above, where $p$ and $q$ are non-empty, increase $c$
appropriately. Notice that there may be many possibilities for a word $r$ with
$\elt{p}\,\elt{q} = \elt{r}$.)

\begin{lemma}
  \label{lem:checkmultiplying}
  There is an algorithm that takes as input a word-hyperbolic structure $\Sigma$ for a semigroup and three words
  $p,q,r \in L(\Sigma)$, and decides whether $\elt{p}\,\elt{q} = \elt{r}$ in time $\bigO((|p| + |q| + |r|)^3)$.
\end{lemma}

\begin{proof}
  The algorithm simply checks whether $p\#_1q\#_2r^\rev \in M(\Sigma)$, and the membership problem for arbitrary
  context-free languages is soluble in cubic time \cite{graham_improved}.
\end{proof}

\section{Word problem}
\label{sec:wordproblem}

This section is dedicated to proving that the uniform word problem for word-hy\-per\-bol\-ic semigroups is soluble in
polynomial time.

As noted in the introduction, the previously-known algorithm required
exponential time \cite[Theorem~3.8]{hoffmann_relatives}. This
motivated Hoffmann \& Thomas to define a narrower notion of
word-hy\-per\-bol\-ic\-ity for monoids that still generalizes
word-hy\-per\-bol\-ic\-ity for groups. By restricting to this version of
word-hy\-per\-bol\-ic\-ity, one recovers automaticity
\cite[Theorem~3]{hoffmann_notions} and an algorithm that runs in time
$\bigO(n\log n)$, where $n$ is the length of the input words
\cite[Theorem~2]{hoffmann_notions}. Although the algorithm described
below is not as efficient as this, the existence of a polynomial-time
solution to the word problem for word-hyperbolic monoids (in the
original Duncan--Gilman sense) diminishes the appeal of the
Hoffmann--Thomas restricted version.

\begin{theorem}
  \label{thm:wordproblem}
  There is an algorithm that takes as input a word-hyperbolic structure $\Sigma$ for a semigroup, where $M(\Sigma)$ is
  defined by a context-free grammar $\Gamma$, and two words $w,w' \in A(\Sigma)^+$ and determines whether
  $\elt{w} = \elt{w'}$ in time polynomial in $|w| + |w'|$ and $|\Gamma|$. More succinctly, the uniform word problem
  for word-hyperbolic semigroups is soluble in polynomial time.
\end{theorem}

\begin{proof}
  By interchanging $w$ and $w'$ if necessary, assume that $|w| \geq
  |w'|$. First, if $|w| = |w'| = 1$, then $w,w' \in A(\Sigma)$ and so
  (since the interpretation map is injective on $A(\Sigma)$), we have
  $\elt{w} = \elt{w'}$ if and only if $w = w'$.

  So assume $|w| \geq 2$. Factorize $w$ as $w = w^{(1)}w^{(2)}$, where
  $w^{(1)} = \lfloor |w|/2\rfloor$. Notice that $\elt{w} = \elt{w'}$ if and
  only if $\elt{w^{(1)}}\,\elt{w^{(2)}} = \elt{w'}$.

  By \fullref{Lemma}{lem:wordproblem} below, there is an algorithm that takes the three words $w^{(1)}$, $w^{(2)}$, and
  $w'$, and the word-hyperbolic structure $\Sigma$, and yields words $u^{(1)}$, $u^{(2)}$, and $u'$ in $L(\Sigma)$
  representing $\elt{w^{(1)}}$, $\elt{w^{(2)}}$, and $\elt{w'}$, of lengths at most $(c+1)|w^{(1)}|^{1+\log(c+1)}$,
  $(c+1)|w^{(2)}|^{1+\log(c+1)}$ and $(c+1)|w'|^{1+\log(c+1)}$, respectively, where $c$ is a constant dependent only on
  $\Sigma$, in time polynomial in $|w^{(1)}|+|w^{(2)}|+|w'|$ and $|\Gamma|$.

  It follows that $\elt{w} = \elt{w'}$ if and only if
  $\elt{u^{(1)}}\,\elt{u^{(2)}} = \elt{u'}$, and, by
  \fullref{Lemma}{lem:checkmultiplying}, this can be checked in time
  cubic in $|u^{(1)}| + |u^{(2)}| + |u'|$, which, by the bounds on
  the lengths of $u^{(1)}$, $u^{(2)}$, and $u'$, is still polynomial
  in the lengths of $w$ and $w'$. Thus the word problem for the
  semigroup described by $\Sigma$ is soluble in polynomial time.
\end{proof}

\begin{lemma}
  \label{lem:wordproblem}
  There is an algorithm that takes as input a word-hyperbolic structure $\Sigma$ for a semigroup, where $M(\Sigma)$ is
  defined by a context-free grammar $\Gamma$, and a word $w \in A(\Sigma)^+$ and outputs a word $u \in L(\Sigma)$ with
  $\elt{w} = \elt{u}$ and $|u| \leq |w|(c+1)|w|^{\log (c+1)}$ (where $c$ is a constant dependent only on $\Sigma$), and
  which takes time polynomial in $|w|$ and $|\Gamma|$.
\end{lemma}

\begin{proof}
  The first step is to convert $\Gamma$ to a quadratic Greibach normal form grammar, so that
  \fullref{Lemma}{lem:multiplyingqgnf} can be applied. This takes time $\bigO(|\Gamma|)^2$.

  Suppose $w = w_1\cdots w_n$, where $w_i \in A \subseteq
  L$. Therefore $w_1,\ldots,w_n$ is a sequence of words in $L$
  whose concatenation represents the same element of the semigroup as
  $w$.

  For the purposes of this proof, the \defterm{total length} of a
  sequence $s_1,\ldots,s_\ell$ of words in $A^*$ is defined to be the
  sum of the lengths of the words $|s_1| + \ldots + |s_\ell|$.

  Consider the following computation, which will form the iterative step
  of the algorithm: suppose there is a sequence of words
  $s_1,\ldots,s_\ell$, each lying in $L(\Sigma)$ and each of length at
  most $t$. Notice that $\ell t$ is an upper bound for the total length
  of this sequence. For $i = 1,\ldots, \lfloor\ell/2\rfloor$, apply
  \fullref{Lemma}{lem:multiplyingqgnf} to compute a word $s'_i \in
  L(\Sigma)$ representing $\elt{s_{2i-1}s_{2i}}$ of length at most
  $c(|s_{2i-1}| + |s_{2i}|) \leq 2ct$. For each $i = 1,\ldots,\lfloor
  \ell/2\rfloor$, this takes $\bigO((|s_{2i-1}| + |s_{2i}|)^5)$ time,
  which is at worst $\bigO((2t)^5)$ time. Therefore the total time used
  is at most $\bigO(\lfloor \ell/2\rfloor(2t)^5)$, which is certainly no
  worse than time $\bigO((\ell t)^5)$. That is, the total time used is
  at worst quintic in the upper bound of the total length of the
  original sequence.

  If $\ell$ is odd, set $s'_{\lceil \ell/2\rceil}$ to be $s_\ell$. (If
  $\ell$ is even, $\lceil \ell/2\rceil = \lfloor\ell/2\rfloor$, so
  $s'_{\lceil \ell/2\rceil}$ has already been computed.) This is purely
  notational; no extra computation is done.

  The result of this computation is a sequence of $\lceil \ell/2\rceil$
  words, each of length at most $2ct$, whose concatenation represents
  the same element of the semigroup as the concatenation of the original
  sequence. The total length of the result is at most $(c+1)\ell t$;
  that is, at most $c+1$ times the total length of the previous
  sequence.

  Apply this computation iteratively, starting with the sequence $w_1,\ldots,w_n$ and continuing until a sequence with
  only one element results. Since each iteration takes a sequence with $\ell$ terms to one with $\lceil\ell/2\rceil$
  terms, there are at most $\lceil\log n\rceil$ iterations. The first iteration of this computation, applied to a sequence
  whose total length is at most $n$, completes in time $\bigO(n^5)$, yielding a sequence of total length at most $n(c+1)$;
  the next iteration completes in time $\bigO((n(c+1))^5)$, yielding a sequence of total length at most $n(c+1)^2$. In
  general the $i$-th iteration completes in time at most $\bigO((n(c+1)^{i-1})^5)$, yielding a sequence of total length at
  most $n(c+1)^i$. So the $\lceil\log n\rceil$ iterations together complete in time at most
  $\bigO((1+\log n)(n(c+1)^{1+\log n})^5)$, since $\lceil\log n\rceil \leq 1+\log n$. (Informally, each iteration yields a
  sequence of roughly half as many words in $L(\Sigma)$ labelling a sequence of arcs that each span a subword twice as
  long as the corresponding terms in the preceding sequence.)

  Applying exponent and logarithm laws,
  \begin{align*}
    n(c+1)^{1+\log n}
    &=n(c+1)(c+1)^{\log n} \\
    &=n(c+1)n^{\log(c+1)} \\
    &=(c+1)n^{1+\log(c+1)},
  \end{align*}
  and so, since $c$ is a constant, the algorithm completes in
  time
  \[
  \bigO(n^{5+5\log(c+1)}\log n),
  \]
  yielding a word in $L(\Sigma)$ of length at most $n(c+1)n^{\log (c+1)}$.
\end{proof}

Interestingly, although \fullref{Theorem}{thm:wordproblem} gives a
polynomial-time algorithm for the word problem for word-hy\-per\-bol\-ic
monoids, the proof does not give a bound on the exponent of the
polynomial, because the constant $c$ of
\fullref{Lemma}{lem:multiplyingqgnf} is dependent on the word-hy\-per\-bol\-ic
structure $\Sigma$. There is thus an open question: does such a bound
actually exist? or can the word problem for hyperbolic semigroups be
arbitrarily hard within the class of polynomial-time problems?

The algorithm described in \fullref{Lemma}{lem:wordproblem} is not
particularly novel. It is similar in outline to that described by
Hoffmann \& Thomas \cite[Lemma~11]{hoffmann_notions} for their
restricted notion of word-hy\-per\-bol\-ic\-ity in monoids. However, the proof
that it takes time polynomial in the lengths of the input words
\emph{is} new.

Hoffmann \& Thomas describe their algorithm in recursive terms: to
find a word in $L(\Sigma)$ representing the same element as $w \in
A^*$, factor $w$ as $w'w''$, where the lengths of $w'$ and $w''$
differ by at most $1$, recursively compute representatives $p'$ and
$p''$ in $L(\Sigma)$ of $\elt{w'}$ and $\elt{w''}$, then compute a
representative for $\elt{w}$ using $p'$ and $p''$. This last step they
prove to take linear time (recall that this only applies for their
restricted notion of word-hy\-per\-bol\-ic\-ity) and to yield a word of length
at most $|p'| + |p''| + 1$, which shows that the whole algorithm takes
time $\bigO(n\log n)$. However, this recursive, `top-down' view of the
algorithm obscures the fact that the overall strategy can be made to
work even for monoids that are word-hy\-per\-bol\-ic in the general
Duncan--Gilman sense. It is through the iterative, `bottom-up' view of
the algorithm presented above that it becomes apparent that the length
increase of \fullref{Lemma}{lem:multiplyingqgnf} remains under control
through the $\log n$ iterations.

\section{Deciding basic properties}
\label{sec:basic}

This section shows that certain basic properties are effectively
decidable for word-hy\-per\-bol\-ic semigroups. First, being a monoid is
decidable:

\begin{algorithm}
\label{alg:monoid}
~

\textit{Input:} An interpreted word-hy\-per\-bol\-ic structure $\Sigma$ for a
semigroup.

\textit{Output:} If the semigroup is a monoid (that is, contains a
two-sided identity), output \Yes\ and a word in $L(\Sigma)$ representing the
identity; otherwise output \No.

\textit{Method:}
\begin{enumerate}

\item For each $a \in A$, construct the context-free language
\begin{equation}
\label{eq:monoidalglang1}
I_a = \gset{ i \in L}{a\#_1i\#_2a \in M }
\end{equation}
and check that it is non-empty. If any of these checks fail, halt and
output \No.

\item For each $a \in A$, choose some $i_a \in I_a$.

\item Iterate the following step for each $a \in A$. For each $b \in
  B$, if $\elt{i_a}\,\elt{b} = \elt{b}\,\elt{i_a} = \elt{b}$, halt and
  output \Yes\ and $i_a$.

\item Halt and output \No.
\end{enumerate}
\end{algorithm}

\begin{proposition}
\fullref{Algorithm}{alg:monoid} outputs \Yes\ and $i$ if and only if
the semigroup defined by $\Sigma$ is a monoid with identity $\elt{i}$.
\end{proposition}

\begin{proof}
Suppose first that \fullref{Algorithm}{alg:monoid} halts with output
\Yes\ and $i$. Then by step~3, $\elt{i}\,\elt{b} = \elt{b}\,\elt{i} =
\elt{b}$ for all $b \in A$. Since $\elt{A}$ generates $S$, it follows
that $s\elt{i} = \elt{i}s = s$ for all $s \in S$ and hence $\elt{i}$
is an identity for $S$.

Suppose now that $S$ is a monoid with identity $e$. Then there is some
word $w \in L$ with $\elt{w} = e$. For every $a \in A$, $\elt{a}e =
\elt{a}$, and so $a\#_1w\#_2a \in M$. Thus $w \in I_a$ for all $a
\in A$ and so each $I_a$ is non-empty. Thus the checks in step~1
succeed and the algorithm proceeds to step~2.

Suppose that $w = w_1\cdots w_n$, where $w_j \in A$ for each
$j=1,\ldots,n$. Then
\begin{align*}
e &= \elt{w} = \elt{w_1\cdots w_{n-1}}\,\elt{w_n} \\
&= \elt{w_1\cdots w_{n-1}} \,\elt{w_n}\,\elt{i_{w_n}} && \text{(by the choice of $i_{w_n} \in I_{w_n}$)} \\
&= e\elt{i_{w_n}} \\
&= \elt{i_{w_n}} && \text{(since $e$ is an identity for $S$).}
\end{align*}
Hence $i_{w_n}$ represents the identity $e$ and so
$\elt{i_{w_n}}\,\elt{b} = \elt{b}\,\elt{i_{w_n}} = \elt{b}$. Thus at
least one of the $i_a$ chosen in step~2 passes the test of step~3
(which guarantees that it represents an identity since $\elt{A}$
generates $S$) and so the algorithm halts at step~3 and outputs
\Yes\ and a word $i_a$ representing the identity.
\end{proof}

\begin{question}
Is there an algorithm that takes as input an interpreted
word-hy\-per\-bol\-ic structure and determines whether the semigroup it
defines contains a zero?

Notice that this cannot be decided using a procedure like
\fullref{Algorithm}{alg:monoid}, or at least not obviously, because
the natural analogue of $I_a$ is
\[
Z_a = \gset{ z \in L}{a\#_1z\#_2z^\rev \in M },
\]
which is naturally defined as the intersection of $M$ and
$\gset{u\#_1v\#_2v^\rev}{u,v \in A^+}$. However, testing the emptiness
of an intersection of context-free languages is in general
undecidable. So using $Z_a$ would, at minimum, require some additional
insight into the kind of context-free languages that can appear as
$M$.
\end{question}

Notice that commutativity is very easy to decide for a word-hy\-per\-bol\-ic
semigroup; one needs to check only that $\elt{ab} = \elt{ba}$ for all
symbols $a,b \in A(\Sigma)$. This is simply a matter of performing a
bounded number of multiplications and checks using
\fullref{Lemmata}{lem:multiplyingqgnf} and \ref{lem:checkmultiplying}.

Green's relation $\gL$ is decidable for automatic semigroups; in
contrast, Green's relation $\gR$ is undecidable, as a corollary of
the fact that right-invertibility is undecidable in automatic monoids
\cite[Theorem~5.1]{kambites_decision}. In contrast, $\gR$ and
$\gL$ are both decidable for word-hy\-per\-bol\-ic semigroups, as a
consequence of $M(\Sigma)$ describing the entire multiplication
table.

\begin{proposition}
There is an algorithm that takes as input an interpreted
word-hy\-per\-bol\-ic structure $\Sigma$ and two words $w,w' \in L(\Sigma)$
and decides whether the elements represented by $w$ and $w'$ are:
\begin{enumerate}
\item $\gR$-related,
\item $\gL$-related,
\item $\gH$-related.
\end{enumerate}
\end{proposition}

\begin{proof}
Let $S$ be the semigroup described by $\Sigma$. The elements $\elt{w}$
and $\elt{w'}$ are $\gR$-related if and only if there exist $s,t \in
S^1$ such that $\elt{w}s = \elt{w'}$ and $\elt{w'}t = \elt{w}$. That
is, $\elt{w} \gR \elt{w'}$ if and only if either $\elt{w} =
\elt{w'}$, or there exist $s,t \in S$ with $\elt{w}s = \elt{w'}$ and
$\elt{w'}t = \elt{w}$. The possibility that $\elt{w} = \elt{w'}$ can
be checked algorithmically by \fullref{Theorem}{thm:wordproblem}. The
existence of an element $s \in S$ such that $\elt{w}s = \elt{w'}$ is
equivalent to the non-emptiness of the language
\[
\gset{v \in L}{w\#_1v\#_2(w')^\rev \in M}.
\]
This context-free language can be effectively constructed and its
non-emptiness effectively decided. Similarly, it is possible to decide
whether there is an element $t \in S$ such that $\elt{w'}t =
\elt{w}$. Hence it is possible to decide whether $\elt{w} \gR
\elt{w'}$.

Similarly, one can effectively decide whether $\elt{w} \gL
\elt{w'}$. Since $\elt{w} \gH \elt{w'}$ if and only if
$\elt{w} \gR \elt{w'}$ and $\elt{w} \gL \elt{w'}$, whether
$\elt{w}$ and $\elt{w'}$ are $\gH$-related is effectively decidable.
\end{proof}

\begin{corollary}
\label{corol:group}
There is an algorithm that takes as input an interpreted
word-hy\-per\-bol\-ic structure and decides whether the semigroup it
describes is a group.
\end{corollary}

\begin{proof}
Suppose the input word-hy\-per\-bol\-ic structure is $\Sigma$ and that it
describes a semigroup $S$. Apply \fullref{Algorithm}{alg:monoid}. If
$S$ is not a monoid, it cannot be a group. Otherwise we know that $S$
is a monoid and we have a word $i \in L(\Sigma)$ that represents its
identity. For each $a \in A(\Sigma)$, check whether $\elt{a} \gR
\elt{i}$ and $\elt{a} \gL \elt{i}$: if all these checks succeed,
then every generator is both right-and left-invertible, and so $S$ is
a group; if any fail, there is some generator that is either not
right- or not left-invertible and so $S$ cannot be a group. Hence
it is decidable whether $\Sigma$ describes a group.
\end{proof}

\begin{question}
Are Green's relations $\gD$ and $\gJ$ decidable for
word-hy\-per\-bol\-ic semigroups?

Note that $\gD$ and $\gJ$ are both undecidable for automatic
semigroups \cite[Theorems~4.1 \&~4.3]{otto_green}.
\end{question}

\section{Being completely simple}
\label{sec:compsimp}

This section shows that it is decidable whether a word-hy\-per\-bol\-ic
semigroup is completely simple. This is particularly useful because a
completely simple semigroup is word-hy\-per\-bol\-ic if and only if its
Cayley graph is a hyperbolic metric space
\cite[Theorem~4.1]{fountain_hyperbolic}, generalizing the equivalence
for groups of these properties for groups.

\begin{definition}
Let $S$ be a semigroup, $I$ and $\Lambda$ be index sets, and $P$ be a
$\Lambda \times I$ matrix over $S$ whose $(\lambda,i)$-th element is
$p_{\lambda,i}$. The Rees matrix semigroup
$\mathcal{M}[S;I,\Lambda;P]$ is defined to be the set $I \times S
\times \Lambda$ with multiplication
\[
(i,g,\lambda)(j,h,\mu) = (i,gp_{\lambda,j}h,\mu).
\]
\end{definition}

Recall that a semigroup is completely simple if it has no proper
two-sided ideals, is not the two-element null semigroup, and contains
a primitive idempotent (that is, an idempotent $e$ such that, for all
idempotents $f$, we have $ef = fe = f \implies e = f$). The version of the
celebrated Rees theorem due to Suschkewitsch
\cite[Theorem~3.3.1]{howie_fundamentals} shows that all completely simple
semigroups are isomorphic to a semigroup $\mathcal{M}[G;I,\Lambda;P]$,
where $G$ is a group and $I$ and $\Lambda$ are finite sets.

Let $A$ be an alphabet representing a generating set for a completely
simple semigroup $\mathcal{M}[G;I,\Lambda;P]$. Define maps $\upsilon : A
\to I$ and $\xi : A \to \Lambda$ by letting $a\upsilon$ and
$a\xi$ be such that $\elt{a} \in \set{a\upsilon} \times G \times
\set{a\xi}$. For the purposes of this paper, we call the pair of maps
$(\upsilon,\xi)$ the \defterm{species} of the completely simple
semigroup. We first of all prove that it is decidable whether a
word-hy\-per\-bol\-ic semigroup is a completely simple semigroup of a
particular species.

\begin{algorithm}
\label{alg:compsimpspeciesdec}
~

\textit{Input:} An interpreted word-hy\-per\-bol\-ic structure $\Sigma$, two
finite sets $I$ and $\Lambda$, and two surjective maps $\upsilon : A(\Sigma)
\to I$ and $\xi : A(\Sigma) \to \Lambda$.

\textit{Output:} If $\Sigma$ describes a completely simple semigroup
of species $(\upsilon,\xi)$, output \Yes; otherwise output
\No.

\textit{Method:} At various points in the algorithm, checks are
made. If any of these checks fail, the algorithm halts and outputs
\No.
\begin{enumerate}

\item For each $i \in I$ and $\lambda \in \Lambda$, construct the
  regular language
\[
L_{i,\lambda} = \gset{a_1\cdots a_n \in L}{a_i \in A, a_1\upsilon = i, a_n\xi = \lambda}.
\]
Check that each
$L_{i,\lambda}$ is non-empty.

\item For each $i,j \in I$ and $\lambda,\mu \in \Lambda$, construct
  the context-free language
\begin{equation}
\label{eq:compsimpalgone}
\gset{u\#_1v\#_2w^\rev \in M}{u \in L_{i,\lambda}, v \in L_{j,\mu}, w \in L - L_{i,\mu}},
\end{equation}
and check that it is empty.

\item For each $i \in I$ and $\lambda \in \Lambda$, choose a word
  $w_{i,\lambda} \in L_{i,\lambda}$ and construct the context-free language
\[
I_{i,\lambda} = \gset{u \in L_{i,\lambda}}{w_{i,\lambda}\#_1 u\#_2 w_{i,\lambda}^\rev \in M}.
\]
Check that each $I_{i,\lambda}$ is non-empty.

\item For each $i \in I$ and $\lambda \in \Lambda$, choose a word
  $u_{i,\lambda} \in I_{i,\lambda}$.

\item For each $a \in A$, $i \in I$, and $\lambda \in \Lambda$, check that
  $\elt{u_{a\upsilon,\lambda}}\,\elt{a} = \elt{a}$ and
  $\elt{a}\,\elt{u_{i,a\xi}} = \elt{a}$.

\item For each $a \in A$, $i \in I$, and $\lambda,\mu \in \Lambda$, calculate
  a word $h_{i,a,\mu,\lambda} \in L$ such that $\elt{h_{i,a,\mu,\lambda}} =
  \elt{u_{i,\mu}}\,\elt{a}\,\elt{u_{i,\lambda}}$.

\item For each $a \in A$, $i \in I$, and $\lambda,\mu \in \Lambda$, check that
  $\elt{h_{i,a,\mu,\lambda}}\,\elt{u_{i,a\xi}} = \elt{u_{i,\mu}}\,\elt{a}$.

\item For each $a \in A$, $i \in I$, and $\lambda,\mu \in \Lambda$, check that
\[
\elt{u_{i,\lambda}}\,\elt{h_{i,a,\mu,\lambda}} = \elt{h_{i,a,\mu,\lambda}}\,\elt{u_{i,\lambda}} = \elt{h_{i,a,\mu,\lambda}}.
\]

\item For each $a \in A$, $i \in I$, and $\lambda,\mu \in \Lambda$, construct the context-free language
\[
V_{i,a,\mu,\lambda} = \gset{v \in L}{h_{i,a,\mu,\lambda}\#_1v\#_2u_{i,\lambda} \in M}
\]
and check that it is non-empty.

\item For each $a \in A$, $i \in I$, and $\lambda,\mu \in \Lambda$, choose
  some $v_{i,a,\mu,\lambda} \in V_{i,a,\mu,\lambda}$ and check that
  $\elt{a}\,\elt{h_{i,a,\mu,\lambda}} = \elt{u_{i,\lambda}}$.

\item Halt and output \Yes.

\end{enumerate}
\end{algorithm}

\fullref{Lemmata}{lem:compsimpalg1} and \ref{lem:compsimpalg2} show
that this algorithm works.

\begin{lemma}
\label{lem:compsimpalg1}
If \fullref{Algorithm}{alg:compsimpspeciesdec} outputs \emph{Yes}, the
semigroup defined by the word-hy\-per\-bol\-ic structure $\Sigma$ is a
completely simple semigroup of species $(\upsilon,\xi)$.
\end{lemma}

\begin{proof}
Let $S$ be the semigroup defined by the input word-hy\-per\-bol\-ic
structure $\Sigma$. Suppose the algorithm output
\textit{Yes}. Then all the checks in steps 1--10 must succeed.

For each $i \in I$ and $\lambda \in \Lambda$, let $T_{i,\lambda} =
\elt{L_{i,\lambda}}$. By the definition of $L_{i,\lambda}$, for each
$a \in A$, the word $a$ lies in $L_{a\upsilon,a\lambda}$. By the check
in step~1, each $T_{i,\lambda}$ is non-empty.

By the check in step~2, for all $i,j \in I$ and $\lambda,\mu \in
\Lambda$, there do not exist $u \in L_{i,\lambda}$, $v \in
L_{j,\mu}$, $w \in L - L_{i,\mu}$ with $\elt{u}\,\elt{v} =
\elt{w}$. That is,
\begin{equation}
\label{eq:compsimpalg1}
T_{i,\lambda}T_{j,\mu} \subseteq T_{i,\mu} \text{ for all $i,j \in I$
  and $\lambda,\mu \in \Lambda$.}
\end{equation}
In particular, $T_{i,\lambda}T_{i,\lambda} \subseteq T_{i,\lambda}$
and so each $T_{i,\lambda}$ is a subsemigroup of $S$.

In each $T_{i,\lambda}$, there is some element that stabilizes some
other element $\elt{w_{i,\lambda}}$ on the right (that is, that
right-multiplies $\elt{w_{i,\lambda}}$ like an identity) by the check in
step~3. In step~4, $u_{i,\lambda}$ is chosen to be such an
element. Let $e_{i,\lambda} = \elt{u_{i,\lambda}}$.

By the check in step~5,
\begin{equation}
\label{eq:compsimpalg2}
e_{a\upsilon,\lambda}\elt{a} = \elt{a} \text{ and }\elt{a}e_{i,a\xi} = \elt{a} \text{ for all $i \in I$ and $\lambda \in
\Lambda$}.
\end{equation}
In step~6, $h_{i,a,\mu,\lambda}$ is calculated for all $i
\in I$, $\lambda,\mu \in \Lambda$, $a \in A$ so that
\begin{equation}
\label{eq:compsimpalg3}
\elt{h_{i,a,\mu,\lambda}} = e_{i,\mu}\elt{a}e_{i,\lambda}.
\end{equation}
By \eqref{eq:compsimpalg1}, $h_{i,a,\mu,\lambda} \in L_{i,\lambda}$.
By the check in step~7,
\begin{equation}
\label{eq:compsimpalg4}
\elt{h_{i,a,\mu,\lambda}}e_{i,a\xi} =
e_{i,\mu}\elt{a} \text{ for all $i \in I$, $\lambda,\mu \in \Lambda$, $a \in
A$.}
\end{equation}

Let $i \in I$ and $\lambda \in \Lambda$. Let $t \in
T_{i,\lambda}$. Then $t = \elt{a_1}\,\elt{a_2}\cdots\elt{a_n}$ for
some $a_k \in A$. Since $a_1a_2\cdots a_n\in L_{i,\lambda}$,
$a_1\upsilon = i$ and $a_n\xi = \lambda$. Then
\begin{flalign*}
&&& \elt{a_1}\,\elt{a_2}\,\elt{a_3}\cdots\elt{a_n} \\
&&={} & e_{i,\lambda}\elt{a_1}\,\elt{a_2}\,\elt{a_3}\cdots\elt{a_n}e_{i,\lambda} & \text{\kern-7em[by \eqref{eq:compsimpalg2}, since $a_1\upsilon = i$ and $a_n\xi = \lambda$]} \\
&&={} & \elt{h_{i,a_1,\lambda,\lambda}}e_{i,a_1\xi}\elt{a_2}\,\elt{a_3}\cdots\elt{a_n}e_{i,\lambda}  & \text{[by \eqref{eq:compsimpalg4}]} \\
&&={} & \elt{h_{i,a_1,\lambda,\lambda}}\,\elt{h_{i,a_2,a_1\xi,\lambda}}e_{i,a_2\xi}\elt{a_3}\cdots\elt{a_n}e_{i,\lambda}  & \text{[by \eqref{eq:compsimpalg4}]} \\
&&={} & \elt{h_{i,a_1,\lambda,\lambda}}\,\elt{h_{i,a_2,a_1\xi,\lambda}}\,\elt{h_{i,a_3,a_2\xi,\lambda}}e_{i,a_3\xi}\cdots\elt{a_n}e_{i,\lambda}  & \text{[by \eqref{eq:compsimpalg4}]} \\
&& &\qquad \vdots \\
&&={} & \elt{h_{i,a_1,\lambda,\lambda}}\,\elt{h_{i,a_2,a_1\xi,\lambda}}\,\elt{h_{i,a_3,a_2\xi,\lambda}}\cdots e_{i,a_{n-1}\xi}\elt{a_n}e_{i,\lambda} \\
&&&& \text{\kern-4em[by repeated use of \eqref{eq:compsimpalg4}]} \\
&&={} & \elt{h_{i,a_1,\lambda,\lambda}}\,\elt{h_{i,a_2,a_1\xi,\lambda}}\,\elt{h_{i,a_3,a_2\xi,\lambda}}\cdots \elt{h_{i,a_n,a_{n-1}\xi,\lambda}}  & \text{[by \eqref{eq:compsimpalg3}]}
\end{flalign*}
Therefore the subsemigroup $T_{i,\lambda}$ is generated by the set of
elements $H_{i,\lambda} = \{\elt{h_{i,a,\mu,\lambda}} : a \in A, \mu \in
\Lambda\}$.

By the check in step~8, for all $i \in I$, $\lambda \in \Lambda$, and
$h \in H_{i,\lambda}$, we have $he_{i,\lambda} = e_{i,\lambda}h = h$. Since
$H_{i,\lambda}$ generates $T_{i,\lambda}$, it follows that
$e_{i,\lambda}$ is an identity for $T_{i,\lambda}$. So each
$T_{i,\lambda}$ is a submonoid of $S$ with identity
$e_{i,\lambda}$. In particular, each $e_{i,\lambda}$ is idempotent.

Let $i \in I$ and $\lambda \in \Lambda$. By the check in step~9,
every element $h \in H_{i,\lambda}$ has a right inverse $h'$ in
$T_{i,\lambda}$. By the check in step~10, $h'h = e_{i,\lambda}$ and so
$h'$ is also a left-inverse for $h$ in $T_{i,\lambda}$. Thus every
generator in $H_{i,\lambda}$ is both right- and left-invertible. Hence
every element of $T_{i,\lambda}$ is both right- and left-invertible
and so $T_{i,\lambda}$ is a subgroup of $S$.

Since $S$ is the union of the various $T_{i,\lambda}$, the semigroup
$S$ is regular and the $e_{i,\lambda}$ are the only idempotents in
$S$. Thus by \eqref{eq:compsimpalg1}, distinct idempotents cannot be
related by the idempotent ordering. Hence all idempotents of $S$ are
primitive. Since $S$ does not contain a zero (since it is the union of
the $T_{i,\lambda}$ and \eqref{eq:compsimpalg1} holds), it is
completely simple by \cite[Theorem~3.3.3]{howie_fundamentals}.
\end{proof}

\begin{lemma}
\label{lem:compsimpalg2}
If semigroup defined by the word-hy\-per\-bol\-ic structure $\Sigma$ is
a completely simple semigroup of species $(\upsilon,\xi)$, then
\fullref{Algorithm}{alg:compsimpspeciesdec} outputs \emph{Yes}.
\end{lemma}

\begin{proof}
Suppose the semigroup $S$ defined by the word-hy\-per\-bol\-ic structure
$\Sigma$ is a completely simple semigroup, with $S =
\mathcal{M}[G;I,\Lambda;P]$. For all $i \in I$ and $\lambda \in
\Lambda$, let $e_{i,\lambda}$ be the identity of the subgroup
$T_{i,\lambda} = \set{i} \times G \times \set{\lambda}$; that is,
  $e_{i,\lambda} = (i,p_{\lambda,i}^{-1},\lambda)$. For each $a \in A$,
  the element $\elt{a}$ has the form $(a\upsilon,g_a,a\xi)$ for some
  $g_a \in G$.

By the definition of multiplication in $S$, the word $a_1\cdots a_n
\in L$ represents an element of $T_{i,\lambda}$ if and only if
$a_1\upsilon = i$ and $a_n\xi = \lambda$. Hence each $L_{i,\lambda}$
must be the preimage of $T_{i,\lambda}$ and map surjectively onto
$T_{i,\lambda}$. In particular, $L_{i,\lambda}$ must be non-empty and
so the checks in step~1 succeed.

For any $i,j \in I$ and $\lambda,\mu \in \Lambda$, we have
$T_{i,\lambda}T_{j,\mu} \subseteq T_{i,\mu}$. Hence if $u \in
L_{i,\lambda}$, $v \in L_{j,\mu}$, and $w \in L$ are such that
$\elt{u}\,\elt{v} = \elt{w}$, then $w \in L_{i,\mu}$. Thus the
language \eqref{eq:compsimpalgone} is empty for all $i,j \in I$ and
$\lambda,\mu \in \Lambda$. Hence all the checks in step~2 succeed.

For any $i \in I$ and $\lambda \in \Lambda$, if $w_{i,\lambda} \in
L_{i,\lambda}$, then $\elt{w_{i,\lambda}} \in T_{i,\lambda}$. Since
$T_{i,\lambda}$ is a subgroup, $\elt{w_{i,\lambda}}e_{i,\lambda} =
\elt{w_{i,\lambda}}$, and $e_{i,\lambda}$ is the unique element of
$T_{i,\lambda}$ that stabilizes $\elt{w_{i,\lambda}}$ on the
right. Thus the language $I_{i,\lambda}$ is non-empty, and consists of
words representing $e_{i,\lambda}$. Hence the checks in step~3
succeed, and the words $u_{i,\lambda}$ chosen in step~4 are such that
$\elt{u_{i,\lambda}} = e_{i,\lambda}$.

In a completely simple semigroup, each idempotent is a left identity
within its own $\gR$-class and $e_{i,a\xi}$ is a right identity
within its own $\gL$-class
\cite[Proposition~2.3.3]{howie_fundamentals}. Hence for each $a \in
A$, $i\in I$, and $\lambda\in\Lambda$, we have
$e_{a\upsilon,\lambda}\elt{a} = \elt{a}$ and $\elt{a}e_{i,a\xi} =
\elt{a}$. Thus the checks in step~5 succeed.

For all $a \in A$, $i \in I$, and $\lambda,\mu \in \Lambda$,
\begin{align*}
& \elt{h_{i,a,\mu,\lambda}}\,\elt{u_{i,a\xi}} \\
={} & e_{i,\mu}\elt{a}e_{i,\lambda}e_{i,a\xi} \\
={} & e_{i,\mu}(a\upsilon,g_a,a\xi)(i,p_{\lambda,i}^{-1},\lambda)(i,p_{a\xi,i}^{-1},a\xi) \\
={} & e_{i,\mu}(a\upsilon,g_ap_{a\xi,i}p_{\lambda,i}^{-1}p_{\lambda,i}p_{a\xi,i}^{-1},a\xi) \\
={} & e_{i,\mu}(a\upsilon,g_a,a\xi) \\
={} & e_{i,\mu}\elt{a}.
\end{align*}
Thus all the checks in step~7 succeed.

For all $a \in A$, $i \in I$, and $\lambda,\mu \in \Lambda$, the element
$\elt{h_{i,a,\mu,\lambda}}$ lies in the subgroup $T_{i,\lambda}$,
whose identity is $e_{i,\lambda}$. Hence all the checks in step~8
succeed. Since all elements of this subgroup are right-invertible,
each language $V_{i,a,\mu,\lambda}$ is non-empty; hence all the checks
in step~9 succeed. Finally, since a right inverse is also a left
inverse in a group, all the checks in step~10 succeed. Therefore the
algorithm reaches step~10 and halts with output \emph{Yes}.
\end{proof}

\begin{theorem}
\label{thm:compsimpdec}
There is an algorithm that takes as input an interpreted
word-hy\-per\-bol\-ic structure $\Sigma$ for a semigroup and decides whether
it is a completely simple semigroup.
\end{theorem}

\begin{proof}
We prove that this problem can be reduced to the problem of deciding
whether the semigroup defined by an interpreted word-hy\-per\-bol\-ic
structure $\Sigma$ is a completely simple semigroup of a particular
species $(\upsilon : A(\Sigma) \to I, \xi : A(\Sigma) \to \Lambda)$.

Let $S$ be the semigroup specified by $\Sigma$. Then $S$ is finitely
generated. Thus we need only consider the problem of deciding whether
$S$ is a finitely generated completely simple semigroup. By the
definition of multiplication in a completely simple semigroup (viewed
as a Rees matrix semigroup), the leftmost generator in a product
determines its $\gR$-class (that is, the $I$-component of the
product) and the rightmost generator in a product determines its
$\gL$-class (that is, the $\Lambda$-component of the product). Thus
there must be at least one generator in each $\gR$- and $\gL$-
class, and hence if $S$ is an $I \times \Lambda$ Rees matrix
semigroup, both $|I|$ and $|\Lambda|$ cannot exceed
$|A(\Sigma)|$.

Thus it is suffices to decide whether $S$ is an $I \times \Lambda$
completely simple semigroup for some fixed choice of $I$ and
$\Lambda$, for one can simply test the finitely many possibilities for
index sets $I$ and $\Lambda$ no larger than $A(\Sigma)$.

One can restrict further, and ask whether $S$ is completely semigroup
of some particular species $(\upsilon : A(\Sigma) \to I, \xi :
A(\Sigma) \to \Lambda)$, for there are a bounded number of
possibilities for the maps surjective $\upsilon$ and $\xi$, so it
suffices to test each one.
\end{proof}

\section{Being a Clifford semigroup}
\label{sec:clifford}

This section is dedicated to showing that being a Clifford semigroup
is decidable for word-hy\-per\-bol\-ic semigroups. Recall the definition of a
Clifford semigroup:

\begin{definition}
Let $Y$ be a [meet] semilattice and let $\gset{G_\alpha}{\alpha \in Y}$
be a collection of disjoint groups with, for all $\alpha, \beta \in Y$
such that $\alpha \geq \beta$, a homomorphism $\phi_{\alpha,\beta} :
G_\alpha \to G_\beta$ satisfying the following conditions:
\begin{enumerate}
\item For each $\alpha \in Y$, the homomorphism $\phi_{\alpha,\alpha}$
  is the identity map.
\item For $\alpha,\beta,\gamma \in Y$ with $\alpha \geq \beta \geq
  \gamma$,
\begin{equation}
\label{eq:cliffmap}
\phi_{\alpha,\gamma} = \phi_{\alpha,\beta}\phi_{\beta,\gamma}.
\end{equation}
\end{enumerate}
The set of elements of the \defterm{Clifford semigroup}
$\cliff[Y;G_\alpha;\phi_{\alpha,\beta}]$ is the union of the disjoint
groups $G_\alpha$. The product of the elements
$s$ and $t$ of $S$, where $s \in G_\alpha$ and $t \in G_\beta$, is
\begin{equation}
\label{eq:cliffmult}
(s\phi_{\alpha,\alpha\wedge\beta})(t\phi_{\beta,\alpha\wedge\beta}),
\end{equation}
which lies in the group $G_{\alpha\wedge\beta}$. [The meet of $\alpha$
  and $\beta$ is denoted $\alpha \wedge \beta$.]
\end{definition}

Notice that if $\cliff[Y;G_\alpha;\phi_{\alpha,\beta}]$ is finitely generated, the semilattice $Y$ must be finitely
generated and thus finite.

Let $A$ be an alphabet representing a generating set for a Clifford
semigroup $\cliff[Y;G_\alpha;\phi_{\alpha,\beta}]$. Define a map $\xi
: A \to Y$ by letting $a\xi$ be such that $\elt{a} \in G_{a\xi}$. For
the purposes of this paper, we call this map $\xi : A \to Y$ the
\defterm{species} of the Clifford semigroup. [Notice that the map
  $\xi$ extends to a unique homomorphism $\xi : A^+ \to Y$.] We first
of all prove that it is decidable whether a word-hy\-per\-bol\-ic semigroup is
a Clifford semigroup of a particular species.

\begin{algorithm}
\label{alg:cliffspeciesdec}
~

\textit{Input:} An interpreted word-hy\-per\-bol\-ic structure $\Sigma$ and a map $\xi : A \to Y$.

\textit{Output:} If $\Sigma$ describes a Clifford semigroup of species
$\xi : A \to Y$, output \Yes; otherwise output \No.

\textit{Method:} At various points in the algorithm, checks are
made. If any of these checks fail, the algorithm halts and outputs
\No.
\begin{enumerate}

\item For each $\alpha \in Y$, construct the regular language
\[
L_\alpha = \gset{w \in L}{w\xi = \alpha}.
\]
(These languages are regular since $L$ is regular, $Y$ is finite, and
the map $\xi : A \to Y$ is known.) Check that each $L_\alpha$ is
non-empty.

\item For each $\alpha,\beta \in Y$, construct the context-free
  language
\begin{equation}
\label{eq:cliffordalg}
\gset{u\#_1v\#_2w^\rev \in M}{u \in L_\alpha,v \in L_\beta, w \in L - L_{\alpha \land \beta}}
\end{equation}
and check that it is empty.

\item For each $\alpha\in Y$, choose some word $w_\alpha \in L_\alpha$
  and construct the context-free language
\[
I_\alpha = \gset{i \in L_\alpha}{w_\alpha\#_1i\#_2w_\alpha^\rev \in M}
\]
and check that $I_\alpha$ is non-empty.

\item For each $\alpha \in Y$, pick some $i_\alpha \in I_\alpha$ and
  check that for all $\alpha,\beta \in Y$,
  $\elt{i_\alpha}\,\elt{i_\beta} = \elt{i_{\alpha\land\beta}}$.

\item For each $a \in A$, check that $\elt{i_{a\xi}}\,\elt{a} =
  \elt{a}\,\elt{i_{a\xi}} = \elt{a}$. For each $\alpha \in Y$ and
  $a \in A$ check that $\elt{a}\,\elt{i_\alpha} =
  \elt{i_\alpha}\,\elt{a}$.

\item For each $\alpha \in Y$ and $a \in A$ such that $a\xi \geq
  \alpha$, construct the context-free language
\[
V_{\alpha,a} = \gset{v \in L_\alpha}{a\#_1 v\#_2 i_\alpha \in M }
\]
and check that $V_{\alpha,a}$ is non-empty.

\item For each $\alpha \in Y$ and $a \in A$ such that $a\xi \geq
  \alpha$, pick some $v_{\alpha,a} \in V_{\alpha,a}$ and check that
  $\elt{v_{\alpha,a}}\,\elt{a} = \elt{i_\alpha}$.

\item Halt and output \Yes.

\end{enumerate}
\end{algorithm}

\fullref{Lemmata}{lem:cliffordalg1} and \ref{lem:cliffordalg2} show
that this algorithm works.

\begin{lemma}
\label{lem:cliffordalg1}
If \fullref{Algorithm}{alg:cliffspeciesdec} outputs \Yes, the
semigroup described by the word-hy\-per\-bol\-ic structure $\Sigma$ is a
Clifford semigroup of species $\xi : A \to Y$.
\end{lemma}

\begin{proof}
Let $S$ be the semigroup defined by the input word-hy\-per\-bol\-ic
structure $\Sigma$. Suppose the algorithm output
\textit{Yes}. Then all the checks in steps 1--7 must succeed.

For each $\alpha \in Y$, let $T_\alpha = \elt{L_\alpha}$. By the check
in step~1, all $T_\alpha$ are non-empty.

By the check in step~2, for every $\alpha,\beta \in Y$, there do not
exist $u \in L_\alpha$, $v \in L_\beta$, $w \in
L-L_{\alpha\land\beta}$ with $\elt{u}\,\elt{v} = \elt{w}$. That is,
$T_\alpha T_\beta \subseteq T_{\alpha\land\beta}$. In particular,
$T_\alpha T_\alpha \subseteq T_\alpha$ and so each $T_\alpha$ is a
subsemigroup of $S$.

In each $T_\alpha$, there is some element that right-multiplies some
other element like an identity by the check in step~3.

For each $\alpha \in Y$, the word $i_\alpha$ represents an element
$e_\alpha$, and the set of elements $E = \gset{e_\alpha}{\alpha \in Y}$
forms a subsemigroup isomorphic to the semilattice $Y$ by the check in
step~4.

By the checks in step~5, for each $a \in A$, the element $e_{a\xi}$
(which, like $\elt{a}$, lies in $T_{a\xi})$) acts like an identity on
$\elt{a}$ (that is, $e_{a\xi}\elt{a} = \elt{a}e_{a\xi} = \elt{a}$), and
every element $e_\alpha$ commutes with $\elt{a}$.

Let $\alpha \in Y$ and $t \in T_\alpha$. Then $t = \elt{a_1}\,\elt{a_2}\cdots\elt{a_n}$
for some $a_i \in A$ with $(a_1a_2\cdots a_n)\xi = \alpha$. Then
\begin{flalign*}
&&& \elt{a_1}\,\elt{a_2}\cdots\elt{a_n} \\
&&={} & e_{a_1\xi}\elt{a_1}e_{a_2\xi}\elt{a_2}\cdots e_{a_n\xi}\elt{a_n} &\text{[by the check in step~6]} \\
&&={} & e_{a_1\xi}e_{a_2\xi}\cdots e_{a_n\xi}\elt{a_1}\,\elt{a_2}\cdots\elt{a_n} &\text{[by the check in step~6]} \\
&&={} & e_{(a_1\xi)\land(a_2\xi)\land \cdots \land (a_n\xi)}\elt{a_1}\,\elt{a_2}\cdots\elt{a_n} &\text{[by the isomorphism of $E$ and $Y$]}\\
&&={} & e_{(a_1a_2\cdots a_n)\xi}\alpha\elt{a_1}\,\elt{a_2}\cdots\elt{a_n} &\text{[by the extension of $\xi$ to $A^+$]} \\
&&={} & e_\alpha\elt{a_1}\,\elt{a_2}\cdots\elt{a_n}.
\end{flalign*}
Thus $t = e_\alpha t$. Similarly $te_\alpha = t$. Hence $e_\alpha$ is
an identity for $T_\alpha$.

For each $\alpha \in Y$ and $a \in A$ with $a\xi \geq \alpha$, there
is an element $\elt{v_{\alpha,a}} \in T_\alpha$ such that
$\elt{v_{\alpha,a}}\,\elt{a} = \elt{a}\,\elt{v_{\alpha,a}} = e_\alpha$
by the checks in steps~6 and~7. Since $T_\alpha$ is generated by
elements $\elt{a}$ such that $a\xi \geq \alpha$, it follows that
$T_\alpha$ is a subgroup of $S$.

Since $L$ is the union of the various $L_\alpha$, the semigroup $S$ is
the union of the various subgroups $T_\alpha$. In particular, $S$ is
regular. Furthermore, the only idempotents in $S$ are the identities
of these subgroups; that is, the elements $e_\alpha$. Since every
$e_\alpha$ commutes with every element of $\elt{A}$, it follows that
all idempotents of $S$ are central. Hence $S$ is a regular semigroup
in which the idempotents are central, and thus is a Clifford
semigroup by \cite[Theorem~4.2.1]{howie_fundamentals}.
\end{proof}

\begin{lemma}
\label{lem:cliffordalg2}
If the semigroup defined by the word-hy\-per\-bol\-ic structure $\Sigma$
is a Clifford semigroup of species $\xi : A \to Y$, then
\fullref{Algorithm}{alg:cliffspeciesdec} outputs \emph{Yes}.
\end{lemma}

\begin{proof}
Suppose the semigroup $S$ defined by the word-hy\-per\-bol\-ic structure
$(A,L,M(L))$ is a Clifford semigroup, with $S =
\cliff[Y;G_\alpha;\phi_{\alpha,\beta}]$. For each $\alpha \in Y$, let
$e_\alpha$ be the identity of $G_\alpha$. The language $L_\alpha$
clearly consists of exactly those words in $L$ that map onto
$G_\alpha$, so $L_\alpha$ is non-empty. Hence the checks in step~1
succeed.

By the definition of multiplication in a Clifford semigroup, $G_\alpha
G_\beta \subseteq G_{\alpha\land\beta}$. Hence if $u \in L_\alpha$, $v
\in L_\beta$, and $w \in L$ are such that $\elt{u}\,\elt{v} = \elt{w}$,
then $w \in L_{\alpha \land\beta}$. Thus the language
\eqref{eq:cliffordalg} is empty for all $\alpha,\beta \in Y$. Hence
all the checks in step~2 succeed.

Let $\alpha \in Y$. For any $w_\alpha \in L_\alpha$, the element
$\elt{w_\alpha}$ lies in the subgroup $G_\alpha$. Thus the language
$I_\alpha$ consists of precisely the words that represent elements of
$G_\alpha$ that right-multiply $w_\alpha$ like an identity. Since
$G_\alpha$ is a subgroup, every element of $I_\alpha$ represents
$e_\alpha$. Since there must be at least one such representative,
$I_\alpha$ is non-empty. Thus every check in step~3
succeeds.

The identities $e_\alpha$ form a subsemigroup isomorphic to the
semilattice $Y$ by the definition of multiplication in a Clifford
semigroup. Thus every check in step~4 succeeds.

Furthermore, every $e_\alpha$ is idempotent and thus central in $S$ by
\cite[Theorem~4.2.1]{howie_fundamentals}, and so every check in step~5
succeeds.

Let $\alpha \in Y$ and $a \in A$ be such that $a\xi \geq \alpha$. Let
$v_{\alpha,a}$ be the word representing
$(\elt{a}\phi_{a\xi,\alpha})^{-1}$. Then
\[
\elt{u_a}\,\elt{v_{\alpha,a}} = \elt{a}(\elt{a}\phi_{a\xi,\alpha})^{-1} = (\elt{a}\phi_{a\xi,\alpha})(\elt{a}\phi_{a\xi,\alpha})^{-1} = e_\alpha.
\]
Hence $v_{\alpha,a} \in V_{\alpha,a}$ and so all the checks in step~6
succeed. Similarly $\elt{v_{\alpha,a}}\,\elt{u_a}$ and so all the
checks in step~7 succeed.

Therefore the algorithm reaches step~8 and halts with output \textit{Yes}.
\end{proof}

\begin{theorem}
\label{thm:cliffdec}
There is an algorithm that takes as input an interpreted
word-hy\-per\-bol\-ic structure $\Sigma$ for a semigroup and decides whether
it is a Clifford semigroup.
\end{theorem}

\begin{proof}
We prove that this problem can be reduced to the problem of deciding
whether the semigroup defined by an interpreted word-hy\-per\-bol\-ic structure
$\Sigma$ is a Clifford semigroup with a particular species $\xi : A(\Sigma)
\to Y$.

Let $S$ be the semigroup specified by $\Sigma$. Then $S$ is
finitely generated. Thus we need only consider the problem of deciding
whether $S$ is a finitely generated Clifford semigroup, whose
corresponding semilattice must therefore also be finitely generated. A
finitely generated semilattice is finite.

So if $S$ is a Clifford semigroup
$\cliff[Y;G_\alpha;\phi_{\alpha,\beta}]$, the semilattice $Y$ must be
a homomorphic image of the free semilattice of rank $|A(\Sigma)|$,
which has $2^{|A(\Sigma)|} - 1$ elements. Thus it is suffices to
decide whether $S$ is a Clifford semigroup for some fixed semilattice
$Y$, for one can simply test the finitely many possibilities for $Y$.

One can restrict further, and ask whether $S$ is a Clifford semigroup
with some fixed semilattice $Y$ and some particular placement of
generators into the semilattice of groups. (That is, with knowledge of
in which group $G_\alpha$ each generator $\elt{a}$ putatively lies,
described by a map $\xi : A(\Sigma) \to Y$. Of course, it is necessary
that $\im \xi$ generates $Y$.) There are a bounded number of
possibilities for the map $\xi$, so it suffices to test each one.
\end{proof}

\section{Being free}
\label{sec:free}

This section shows that it is decidable whether a word-hy\-per\-bol\-ic
semigroup is free. The following technical lemma, which is possibly of
independent interest, is necessary.

\begin{lemma}
\label{lem:freecheck}
There is an algorithm that takes as input an alphabet $A$, a symbol
$\#_2$ not in $A$, and a context-free grammar $\Gamma$ defining a
context-free language $L(\Gamma)$ that is a subset of $A^*\#_2A^*$,
and decides whether $L(\Gamma)$ contains a word $x\#_2w^\rev$ where $x
\neq w$.
\end{lemma}

\begin{proof}
Suppose $\Gamma = (N,A \cup \set{\#_2},P,O)$. [Here, $N$ is the set of
  non-terminal symbols, $A \cup \set{\#_2}$ is of course the set of
  terminal symbols, $P$ the set of productions, and $O \in N$ is the
  start symbol.] Since $L(\Gamma)$ does not contain the empty word (since
every word in $L(\Gamma)$ lies in $A^*\#_2A^*$), assume without loss that
$\Gamma$ contains no useless symbols or unit productions
\cite[Theorem~4.4]{hopcroft_automata}.

Let
\[
N_\# = \gset{M \in N}{(\exists p,q \in A^*)(M \derives p\#_2q)}.
\]
Notice that if $M \to p$ is a production in $P$ and $M \in N - N_\#$,
then every non-terminal symbol appearing in $p$ also lies in $N -
N_\#$. [This relies on there being no useless symbols in $\Gamma$,
  which means that every other non-terminal in $P$ derives some
  terminal word.] For this reason, it is easy to compute $N_\#$.

Suppose that $M \derives uMv$ for some $M \in N - N_\#$ and $u,v \in
(A \cup \set{\#_2})^*$. Then $u$ and $v$ cannot contain $\#_2$ since $M
\in N - N_\#$. Since there are no unit productions in $P$, at least
one of $u$ and $v$ is not the empty word. Since $M$ is not a useless
symbol, it appears in some derivation of a word $w\#_2x^\rev \in
L(\Gamma)$. Pumping the derivation $M \derives uMv$ yields a word
$w'\#_2(x')^\rev$ where exactly one of $w' = w$ or $x' = x$ holds,
since the extra inserted $u$ and $v$ cannot be on opposite sides of
the symbol $\#_2$ since $M \in N - N_\#$. Hence either $w \neq x$ or
$w' \neq x'$. Hence in this case $L(\Gamma)$ does contain a word of the
given form.

Since it is easy to check whether there is a non-terminal $M \in
N-N_\#$ with $M \derives uMv$, we can assume that no such non-terminal
exists. Therefore any non-terminal $M \in N-N_\#$ derives only
finitely many words (since any derivation starting at $M$ can only
involve non-terminals in $N - N_\#$ and by assumption no such
non-terminal can appear twice in a given derivation). These words can
be effectively enumerated. Let $M \in N-N_\#$ and let $w_1,\ldots,w_n$
be all the words that $M$ derives. Replacing a production $S \to pMq$
by the productions $S \to pw_1q$, $S \to pw_2q$, \ldots, $S \to pw_nq$
does not alter $L(\Gamma)$. Iterating this process, we eventually
obtain a grammar $\Gamma$ where no non-terminal symbol in $N - N_\#$
appears on the right-hand side of a production. Thus all symbols in $N
- N_\#$ can be eliminated and we now have a grammar $\Gamma$ with $N =
N_\#$.

Every production is now of the form $M \to pSq$ or $M \to p\#_2q$,
where $p,q \in A^*$ and $S \in N$. [There can be only one non-terminal
  on the right-hand side of each production, since otherwise some
  terminal word would contain two symbols $\#_2$, which is impossible.]

We are now going to iteratively define a map $\phi : N \to \fgrp{A}$,
where $\fgrp{A}$ denotes the free group on $A$, which we will identify
with the set of reduced words on $A \cup A^{-1}$. First, define $O\phi
= \emptyword$. Now, iterate through the productions as follows. Choose
some production $M \to pSq^\rev$ such that $M\phi$ is already
defined. Let $z = p^{-1}(M\phi)q \in \fgrp{A}$. If $S\phi$ is
undefined, set $S\phi = z$. If $S\phi$ is defined, check that $S\phi$
and $z$ are equal; if they are not, halt: $L(\Gamma)$ does contain
words $w\#_2x^\rev$ with $w \neq x$.

To see this, suppose $S\phi =z$ and consider the sequence of
productions that gave us the original value of $S\phi$:
\[
O \to u_1S_1v_1^\rev, S_1 \to u_2S_2v_2^\rev, \ldots, S_k \to u_kSv_k^\rev,
\]
which implies that $S\phi = (u_1u_2\cdots u_k)^{-1}v_1v_2\cdots v_k$, and the sequence that gave us $M\phi$:
\[
O \to p_1M_1q_1^\rev, M_1 \to p_2M_2q_2^\rev, \ldots, M_k \to p_lMq_l^\rev,
\]
which implies that $M\phi = (p_1p_2\cdots p_l)^{-1}q_1q_2\cdots
q_l$. Choose $r,s \in A^*$ such that $S \derives r\#_2s^\rev$. Then
$L(\Gamma)$ contains both both $u_1\cdots u_kr\#_2s^\rev
v_k^\rev\cdots v_1^\rev$ and (recalling that $M \to pSq^\rev$ is a
production) $p_1\cdots p_lpr\#_2s^\rev q^\rev q_l^\rev\cdots
q_1$. Suppose $u_1\cdots u_kr = v_1\cdots v_ks$ and $p_1\cdots p_lpr =
q_1\cdots q_lqs$. Then $S\phi = (u_1\cdots u_k)^{-1}v_1\cdots v_k =
rs^{-1} = (p_1\cdots p_lp)^{-1}q_1\cdots q_lq = p^{-1}(M\phi)q = z$,
which is a contradiction.

Once we have iterated through all the productions of the form $M \to
pSq^\rev$, iterate through the productions of the form $M \to
p\#_2q^\rev$, and check that $p^{-1}(M\phi)q$. If this check fails,
halt: $L(\Gamma)$ does contain words $w\#_2x^\rev$ with $w \neq x$;
the proof of this is very similar to the previous paragraph.

Finally, notice that if the iteration through all the productions
completes with all the checks succeeding, a simple induction on
derivations, using the values of $M\phi$, shows that all words
$w\#_2x^\rev \in L(\Gamma)$ are such that $w = x$.
\end{proof}

\begin{algorithm}
\label{alg:free}
~

\textit{Input:} An interpreted word-hy\-per\-bol\-ic structure $\Sigma$.

\textit{Output:} If $\Sigma$ describes a free semigroup, output \textit{Yes}; otherwise output \textit{No}.

\textit{Method:}
\begin{enumerate}

\item For each $a \in A$, iterate the following:

\begin{enumerate}

\item Construct the context-free language
\[
D_a = \gset{uv}{u\#_1v\#_2a^\rev \in M}.
\]

\item Check whether $D_a$ is empty. If it is empty, proceed to the next
  interation. If it is non-empty, choose some word $d_a \in D_a$. If $d_a$
  contains the letter $a$, halt and output \No. If $d_a$ does not
  contain the letter $a$, define the rational relations
\begin{align*}
\drel{Q}_L ={}& \bigl(\set{(a,d_a)} \cup \gset{(b,b)}{b \in A - \set{a}}\bigr)^+ \\
\drel{Q}_M ={}& \bigl(\set{(a,d_a)} \cup \gset{(b,b)}{b \in A - \set{a}}\bigr)^+\#_1\\
&\qquad\bigl(\set{(a,d_a)} \cup \gset{(b,b)}{b \in A - \set{a}}\bigr)^+\#_2\\
&\qquad\qquad\bigl(\set{(a,d_a^\rev)} \cup \gset{(b,b)}{b \in A - \set{a}}\bigr)^+.
\end{align*}
Modify $\Sigma$ as follows: replace $A$ by $A - \set{a}$; replace $L$ by
$L \circ \drel{Q}_L$; and replace $M$ by $M \circ \drel{Q}_M$, and
proceed to the next iteration.

\end{enumerate}

\item If $L \neq A^+$, halt and output \No.

\item Define the rational relation
\[
\drel{P} = \gset{(a,a)}{a \in A\} \cup \{(\#_1,\emptyword),(\#_2,\#_2)}.
\]
Let $N = M \circ \drel{P}$. Using the method of
\fullref{Lemma}{lem:freecheck}, check whether $N$ contains any word of
the form $x\#_2w^\rev$ with $x \neq w$. If so, halt and ouput
\No. Otherwise, halt and output \Yes.

\end{enumerate}
\end{algorithm}

\fullref{Lemmata}{lem:free0} to \ref{lem:free2} show that this
algorithm works.

\begin{lemma}
\label{lem:free0}
If $\Sigma$ is a word-hy\-per\-bol\-ic structure for a semigroup $S$, then
the replacement $\Sigma$ produced in step~1(b) is also a
word-hy\-per\-bol\-ic structure for a semigroup $S$.
\end{lemma}

\begin{proof}
If the language $D_a$ is non-empty, then any word $w \in D_a$ is such
that $\elt{w} = \elt{a}$. In particular, $\elt{d_a} =
\elt{a}$. Furthermore, since $d_a \in (A-\set{a})^*$, we see that
$\elt{a}$ is a redundant generator. The rational relation $\drel{Q}_L$
relates any word in $A^+$ to the corresponding word in $(A-\set{a})^+$
with all instances of the symbol $a$ replaced by the word $d_a$. The
rational relation $\drel{Q}_M$ relates any word in $A^+\#_1A^+\#_2A^+$
to the corresponding word in
$(A-\set{a})^+\#_1(A-\set{a})^+\#_2(A-\set{a})^+$ with all instances of the
symbol $a$ \emph{before} $\#_2$ replaced by the word $d_a$ and all
instances of the symbol $a$ \emph{after} $\#_2$ replaced by the word
$d_a^\rev$. Hence
\[
M\circ\drel{Q}_M \subseteq
(L\circ\drel{Q}_L)\#_1(L\circ\drel{Q}_L)\#_2(L\circ\drel{Q}_L)^\rev.
\]
Since application of rational relations preserves regularity and
context-freedom, $L\circ\drel{Q}_L$ is regular and $M \circ\drel{Q}_M$
is context-free. Finally, since $\elt{a} = \elt{d_a}$, we see that
$L\circ \drel{Q}_L$ maps onto $S$, and similarly $M\circ\drel{Q}_M$
describes the multiplication of elements of $S$ in terms of
representatives in $L$.
\end{proof}

\begin{lemma}
\label{lem:free1}
If \fullref{Algorithm}{alg:free} outputs \emph{Yes}, the semigroup
defined by the word-hy\-per\-bol\-ic structure $\Sigma$ is a free semigroup.
\end{lemma}

\begin{proof}
The algorithm can only halt with output \Yes\ in step~3, so the
algorithm must pass step~2 as well. Hence the language of
representatives is $A^+$. Let $S$ be the semigroup defined by $L$
and let $\phi : A^+ \to S$ be an interpretation.

Suppose for \textit{reductio ad absurdum} that $\phi$ is not
injective. Then there are distinct words $u,v \in A^*$ such that
$u\phi = v\phi$. Since $\phi|_A$ is injective by definition, at least
one of $u$ and $v$ has length $2$ or more. Interchanging $u$ and $v$
if necessary, assume $|u| \geq 2$. So $u = u'u''$, where $u'$ and
$u''$ are both non-empty. Since $L = A^+$, we have $u',u'' \in L$ and
so $u'\#_1u''\#_2v^\rev \in M$. Hence $u\#_2v^\rev \in N$. But since
the algorithm outputs \Yes\ at step~3, there is no word $x\#_2w^\rev
\in M$ with $x \neq w$. This is a contradiction and so $\phi$ is
injective.

So $\phi : A^+ \to S$ is an isomorphism and so $S$ is free.
\end{proof}

\begin{lemma}
\label{lem:free2}
If the word-hy\-per\-bol\-ic structure $\Sigma$ defines a free semigroup,
\fullref{Algorithm}{alg:free} outputs \emph{Yes}.
\end{lemma}

\begin{proof}
Let $B^+$ be the semigroup defined by $\Sigma$. Let $\phi : A^+ \to
B^+$ be an interpretation. Since elements of $B$ are indecomposable,
$B \subseteq A\phi$.

In step~1, the algorithm iterates through each $a \in A$. For each $a
\in B\phi^{-1} \subseteq A$, since $a\phi$ is indecomposable, the
language $D_a$ is empty and the algorithm moves to the next iteration.

Let $a \in A - B\phi^{-1}$. Then $a\phi$ has length (in $B^+$) at
least two and so is decomposable. Hence there exist $u,v \in L$ such
that $uv \in D_a$. Furthermore, since $u\phi$ and $v\phi$ must be
shorter (in $B^+$) than $a\phi$, neither $u$ nor $v$ can include the
letter $a$. Hence the replacement of $\Sigma$ described in step~1(b)
takes place. Since this occurs for all $a \in A-B\phi^{-1}$, at the
end of step~1 we have a word-hy\-per\-bol\-ic structure $\Sigma$ with $A =
B\phi^{-1}$. Since $\phi|_A$ is injective, $\phi|_A$ must be a
bijection from $A$ to $B$. Hence the homomorphism $\phi : A^+ \to B^+$
must be an isomorphism, and so $L= A^+$; thus the check in step~2 is
successful. Therefore
\[
M = \gset{u\#_1v\#_2(uv)^\rev}{u,v \in A^+}
\]
and so
\[
M \circ \drel{P} = \gset{w\#_2w^\rev}{w \in A^+}.
\]
Thus the check in step~3 is successful and the algorithm terminates
with output \Yes.
\end{proof}

Thus, from \fullref{Lemmata}{lem:free1} and \ref{lem:free2}, we
obtain the decidability of freedom for word-hy\-per\-bol\-ic semigroups:

\begin{theorem}
\label{thm:freedec}
There is an algorithm that takes as input an interpreted
word-hy\-per\-bol\-ic structure $\Sigma$ for a semigroup and decides whether
it is a free semigroup.
\end{theorem}

\section{Open problems}
\label{sec:questions}

This conclusing section lists some important question regarding
decision problems for word-hy\-per\-bol\-ic semigroups.

\begin{question}
Is there an algorithm that takes as input an interpreted
word-hy\-per\-bol\-ic structure for a semigroup and decides whether that
semigroup is (a) regular, (b) inverse?

Whether these properties are decidable for automatic semigroups is
currently unknown.
\end{question}

\begin{question}
Is there an algorithm that takes as input an interpreted
word-hy\-per\-bol\-ic structure for a semigroup and decides whether that
semigroup is left-/right-/two-sided-cancellative?

Cancellativity and left-cancellativity are undecidable for automatic
semigroups \cite{c_cancundec}. Right-cancellativity is, however,
decidable \cite[Corollary~3.3]{kambites_decision}.
\end{question}

\begin{question}
Is there an algorithm that takes as input an interpreted
word-hy\-per\-bol\-ic structure for a semigroup and decides whether that
semigroup is finite?

The equivalent question for automatic semigroups is easy: one takes an
automatic structure, effectively computes an automatic structure with
uniqueness, and checks whether its regular language of representatives
is finite. However, this approach cannot be used for word-hy\-per\-bol\-ic
semigroups, because there exist word-hyperbolic semigroups that do not
admit word-hy\-per\-bol\-ic structures with uniqueness indeed, they may not
even admit regular languages of unique normal forms \cite[Examples~10
  \&~11]{cm_wordhypunique}.
\end{question}

\begin{question}
Is there an algorithm that takes as input an interpreted
word-hy\-per\-bol\-ic structure for a semigroup and decides whether that
semigroup admits a word-hy\-per\-bol\-ic structure with uniqueness? (That
is, where the language of representatives maps bijectively onto the
semigroup.) If so, it is possible to compute a word-hy\-per\-bol\-ic
structure with uniqueness in this case?
\end{question}

\bibliography{automaticsemigroups,languages,presentations,rewriting,semigroups,c_publications,\jobname}
\bibliographystyle{alphaabbrv}

\end{document}